\documentclass[10pt]{amsart}
\usepackage{amsthm, amsfonts, amssymb, amsmath, amscd}
\usepackage[all]{xy}
\usepackage{graphicx}
\usepackage{lmodern}
\usepackage[margin=3.5cm]{geometry}

\newcommand{\To}{\rightarrow}

\newcommand{\into}{\hookrightarrow}
\newcommand{\onto}{\twoheadrightarrow}

\newcommand{\e}{\varepsilon}
\newcommand{\noin}{\noindent}

\newcommand{\C}{\mathbb{C}}
\newcommand{\R}{\mathbb{R}}

\newcommand{\Z}{\mathbb{Z}}

\newcommand{\Cstar}{\mathrm{C^*}}
\newcommand{\Cred}{\mathrm{C^*_r}}
\newcommand{\G}{\Gamma}

\newcommand{\KH}{\mathcal{K}}
\newcommand{\U}{\mathrm{U}}
\newcommand{\M}{\mathrm{M}}
\newcommand{\GL}{\mathrm{GL}}

\newcommand{\sr}{\mathrm{sr}\:}
\newcommand{\tsr}{\mathrm{tsr}\:}

\newcommand{\bsr}{\mathrm{bsr}\:}
\newcommand{\csr}{\mathrm{csr}\:}

\newcommand{\gsr}{\mathrm{gsr}\:}

\newcommand{\spec}{\mathrm{sp}}

\newcommand{\dom}{\mathrm{dom}}

\newcommand{\K}{\mathcal{K}}

\newcommand{\id}{\mathrm{id}}

\newcommand{\Lg}{\mathrm{Lg}}

\newtheorem{thm}{Theorem}[section]
\newtheorem{lem}[thm]{Lemma}
\newtheorem{cor}[thm]{Corollary}
\newtheorem{prop}[thm]{Proposition}

\theoremstyle{definition}
\newtheorem{defn}[thm]{Definition}
\newtheorem{rem}[thm]{Remark}
\newtheorem{ex}[thm]{Example}

\newtheorem{prob}[thm]{Problem}

\newtheorem*{ack}{Acknowledgments}

\begin{document}
\title{Homotopical stable ranks for Banach algebras}
\author{Bogdan Nica}
\address{\newline Department of Mathematics and Statistics, University of Victoria, Victoria (BC), Canada V8W 3R4}
\date{March 4, 2011}

\begin{abstract}
\noin The connected stable rank and the general stable rank are homotopy invariants for Banach algebras, whereas the Bass stable rank and the topological stable rank should be thought of as dimensional invariants. This paper studies the two homotopical stable ranks, viz. their general properties as well as specific examples and computations. The picture that emerges is that of a strong affinity between the homotopical stable ranks, and a marked contrast with the dimensional ones.
\end{abstract}

\thanks{Supported by a Postdoctoral Fellowship from the Pacific Institute for the Mathematical Sciences (PIMS)}
\maketitle

\section{Introduction}
We owe to Bass \cite{Bas64} the first notion of stable rank. Many other stable ranks have appeared since then, and it is customary to refer to this original stable rank as the \emph{Bass stable rank}. The Bass stable rank is a purely algebraic - in fact, ring-theoretic - notion. A topological relative of the Bass stable rank, the \emph{topological stable rank}, was introduced by Rieffel \cite{Rie83} in the context of Banach algebras. For $\Cstar$-algebras, the Bass stable rank and the topological stable rank coincide. Furthermore, they can be interpreted as ``noncommutative'' notions of dimension, due to the fact that the Bass / topological stable rank of $C(X)$, where $X$ is a compact Hausdorff space, is $\lfloor \tfrac{1}{2} \dim X \rfloor+1$. 

While investigating the topological stable rank, Rieffel \cite{Rie83} was prompted to define two other stable ranks: the \emph{connected stable rank}, and the \emph{general stable rank}. The first one is, again, topological, whereas the second one is algebraic. Among the four stable ranks we mentioned, the general stable rank is the least studied and arguably the hardest to compute. Yet it is also one of the most natural stable ranks. For the general stable rank starts from the regrettable fact that not all stably free modules are free (to paraphrase J.F. Adams \cite[p.2]{Ada}), and quantifies the property that stably free modules of big enough rank are free.

We initially embarked on a study of the general stable rank for Banach algebras. But very soon, a productive analogy with the connected stable rank emerged. Due to their distinctive feature of being homotopy invariants, the connected and the general stable ranks are collectively referred to as \emph{homotopical} stable ranks in what follows. This terminology is not only meant to mark the analogy between the connected and the general stable ranks, but also to emphasize the contrast with the \emph{dimensional} stable ranks, namely the Bass and the topological stable ranks. Thus, the paper ended up as a comparative study - homotopical stable ranks versus dimensional stable ranks. The emphasis is clearly on the former; in fact, many of the new results, though not all, concern the general stable rank. 

Let us describe the contents of the paper. Section 3 is devoted to definitions and basic facts on the quartet of stable ranks; we also illustrate the two extreme cases, stable rank one and infinite stable rank. In Section 4, we discuss the homotopy invariance of the homotopical stable ranks. The computation of the homotopical stable ranks for $C(X)$, the subject of Section 5, turns out to be much harder than the computation of the dimensional stable ranks. While for the connected stable rank we still have a useful cohomological criterion (Theorem~\ref{csrdim}), the situation for the general stable rank is rather unsatisfactory. Using some detailed information about the homotopy groups of unitary groups, we succeed in computing the general stable rank for $C(X)$ when $X$ is a sphere (Proposition~\ref{gsrspheres}) - thereby providing the first non-trivial computation in this direction. However, we do not know how to compute the general stable rank for $C(X)$ when $X$ is, say, a torus (Problem~\ref{gsrtori}). In turn, computing the homotopical stable ranks for $C(X)$ is crucial for computing the homotopical stable ranks for commutative Banach algebras: as we point out in Section 6, the homotopical stable ranks are invariant under the Gelfand transform (Theorem~\ref{Gelfand csrgsr}). Subsequent sections consider the behavior of the homotopical stable ranks under various operations: matrix algebras (Section 7), quotients (Section 8) and morphisms with dense image (Section 9), inductive limits (Section 11), and extensions (Section 12). The goal of the final section is the computation of the homotopical stable ranks for tensor products of extensions of $\mathcal{K}$ by commutative $\Cstar$-algebras. This extended example builds on Nistor's computation of the dimensional stable ranks for such $\Cstar$-algebras (\cite{Nis86}), as well as on a number of properties established throughout the paper.

It is often said that stable ranks are related to K-theory. Swan's problem, discussed at length in Section 10, is concerned with the following specific aspect: having in mind that K-theory is invariant across dense and spectrum-preserving morphisms, is the same true for stable ranks? Namely, are stable ranks invariant across dense and spectrum-preserving morphisms? While this problem is still open for the dimensional stable ranks, a positive answer for the the homotopical stable ranks is given in Theorem~\ref{Swancsrgsr}. This adds further support to the idea that the favor of being related to K-theory falls upon the homotopical stable ranks, rather than the dimensional ones. The first hint that the homotopical stable ranks are closer to K-theory than the dimensional stable ranks is, of course, the invariance under homotopy. Although we do not discuss this topic here, we would like to mention one more argument: the homotopical stable ranks provide the finest control in unstable K-theory, whereas the estimates involving the dimensional stable ranks are derived as secondary estimates. 

\begin{ack}
A preliminary version of this paper appeared in my dissertation \cite{Nic09}, written under the guidance of Guoliang Yu. I am most grateful to him for all the support and advice. I also thank Marc Rieffel and Leonard R. Rubin for some useful e-conversations. 
\end{ack}

\section{Conventions. Notations} For simplicity, we work with (complex) Banach algebras only. However, our Banach algebra setting could be safely enlarged to the context of Fr\'{e}chet algebras having an open group of invertibles. Throughout most of the paper, Banach algebras and their (continuous) morphisms are assumed to be unital. Starting with Section~\ref{Inductive limits}, we allow for algebras and morphisms which are not necessarily unital.

Let $A$ be a Banach algebra. The component of the identity in $\GL_n(A)$ is denoted $\GL_n^0(A)$. We often write $(a_i)\in A^n$, or $\underline{a}\in A^n$, to mean an $n$-tuple $(a_1,\dots,a_n)$. By $\mathcal{P}(A)$ we denote the set of isomorphism classes of finitely generated (f.g.) projective right $A$-modules; $\mathcal{P}(A)$ is an abelian monoid under direct sum. The K-theory of $A$ is understood in the topological sense; however, we find it more useful to adopt the algebraic picture for $K_0(A)$, namely, as the Grothendieck group of $\mathcal{P}(A)$.

Topological spaces are assumed to be Hausdorff and non-empty. As usual, $I^d$, $S^d$, and $T^d$ denote the $d$-dimensional cube, the $d$-dimensional sphere, and the $d$-dimensional torus.


The final piece of convention is the following notational abuse: for a morphism $\phi: A\To B$, we use $\phi$ to denote a number of natural maps induced by $\phi$, e.g., we write $\phi: \M_n(A)\to \M_n(B)$ and $\phi: A^n\to B^n$.

\section{Stable ranks} 
\subsection{Definitions}
Let $A$ be a (unital) Banach algebra. Consider, for each $n\geq 1$, the collection of $n$-tuples that generate $A$ as a left ideal:
\[\Lg_n(A)=\{(a_1,\dots,a_n):\; Aa_1+\ldots + Aa_n=A\}\subseteq A^n\]
Elements of $\Lg_n(A)$ are called \emph{(left) unimodular $n$-tuples}. A simple, but important, observation is that $\Lg_n(A)$ is open in $A^n$. Indeed, let $(a_i)\in \Lg_n(A)$. Then $\sum b_ia_i=1$ for some $(b_i)\in A^n$; let $U\subseteq A$ be a neighborhood of $0$ for which $\sum b_i U\subseteq A^\times -1$. We have $(a'_i)\in \Lg_n(A)$ whenever $a_i'\in a_i+U$, as $\sum b_ia'_i\in \sum b_ia_i +\sum b_i U=1+\sum b_i U\subseteq A^\times$.

There is an action of $\GL_n(A)$ on $\Lg_n(A)$, given by left-multiplying the transpose of a unimodular $n$-tuple by an invertible matrix: $(\alpha, \underline{a})\mapsto\alpha\cdot\underline{a}^T$ for $\alpha\in \GL_n(A)$ and $\underline{a}\in\Lg_n(A)$.

\begin{defn}\label{definestableranks} Let $A$ be a unital Banach algebra. Then:
\begin{itemize}
\item[(bsr)] the \emph{Bass stable rank} of $A$ is the least $n\geq 1$ such that the following holds: if $(a_1,\dots, a_{n+1})\in \Lg_{n+1}(A)$, then $(a_1+x_1a_{n+1}, \dots, a_n+x_na_{n+1})\in \Lg_n(A)$ for some $(x_i)\in A^n$

\smallskip
\item[(tsr)] the \emph{topological stable rank} of $A$ is the least $n\geq 1$ such that $\Lg_n(A)$ is dense in $A^n$

\smallskip
\item[(csr)] the \emph{connected stable rank} of $A$ is the least $n\geq 1$ such that $\GL^0_m(A)$ acts transitively on $\Lg_m(A)$ for all $m\geq n$

\smallskip
\item[(csr$'$)] the \emph{connected stable rank} of $A$ is the least $n\geq 1$ such that $\Lg_m(A)$ is connected for all $m\geq n$

\smallskip
\item[(gsr)] the \emph{general stable rank} of $A$ is the least $n\geq 1$ such that $\GL_m(A)$ acts transitively on $\Lg_m(A)$ for all $m\geq n$

\smallskip
\item[(gsr$'$)] the \emph{general stable rank} of $A$ is the least $n\geq 1$ such that the following holds: for all $m\geq n$, if $P$ is a right $A$-module satisfying $P\oplus A\simeq A^m$ then $P\simeq A^{m-1}$

\smallskip
\end{itemize}
The above stable ranks of $A$ are respectively denoted $\bsr A$, $\tsr A$, $\csr A$, $\gsr A$. The generic $\sr A$ stands for any one of these.
\end{defn}

\begin{rem}
Several comments concerning the above definition are in order.

a) The action of $\GL^0_n(A)$ on $\Lg_n(A)$ has open orbits. Indeed, fix $(a_i)\in\Lg_n(A)$. Pick $(b_i)\in A^n$ with $\sum b_ia_i=1$, and let $U$ be a neighborhood of $0$ such that $1_n+(u_ib_j)_{1\leq i,j\leq n}$ is in $\GL^0_n(A)$ for all $u_i\in U$. As $1_n+(u_ib_j)_{i,j}$ takes $(a_i)$ to $(a_i+u_i)$, it follows that $(a_i)+U^n$ is included in the $\GL^0_n(A)$-orbit of $(a_i)$. 

Therefore $\Lg_n(A)$ is connected if and only if $\GL^0_n(A)$ acts transitively on $\Lg_n(A)$. This justifies the equivalence of (csr) and (csr$'$) (cf. \cite[Cor.8.4]{Rie83}). Typically, we use (csr) when we pursue the analogy with the general stable rank; (csr$'$), on the other hand, is usually more convenient when the connected stable rank is considered on its own.

b) The equivalence between (gsr) and (gsr$'$) is proved in \cite[Prop.10.5]{Rie83}. It should be stressed that the general stable rank is - just like the Bass stable rank - a ring-theoretic notion, and that many of the facts appearing herein hold for rings, or can be adapted to a ring-theoretic context.

c) Under the action of $\GL_n(A)$ on $\Lg_n(A)$, a matrix takes the unimodular $n$-tuple $(0,\dots, 0,1)$ to the last column of the matrix. Hence $\GL_n(A)$, respectively $\GL^0_n(A)$, acts transitively on $\Lg_n(A)$ if and only if each unimodular $n$-tuple is the last column of some matrix in $\GL_n(A)$, respectively $\GL^0_n(A)$.

d)  If $A$ is non-unital, then $\sr A$ is defined to be $\sr A^+$, where $A^+$ is the unitization of $A$.

e) It is easy to see that $\sr A\oplus B= \max\{\sr A,\sr B\}$ whenever $A$ and $B$ are unital Banach algebras. In particular, if $A^+$ denotes the Banach algebra obtained by adding a new unit to a unital Banach algebra $A$, then $\sr A^+=\sr A$ since $\sr \C=1$.

f) We put $\sr A=\infty$ whenever there is no integer $n$ satisfying the required stable rank condition.


g) Definition~\ref{definestableranks} actually describes the \emph{left} stable ranks. The right counterpart for each left stable rank is defined with respect to the right unimodular $n$-tuples $\{(a_1,\dots,a_n):\; a_1A+\ldots + a_nA=A\}$. Clearly, for Banach $*$-algebras there is no difference between left and right stable ranks. For general Banach algebras, the Bass stable rank is left-right symmetric (\cite{Vas71}; see also \cite[Prop.11.3.4]{MR87}), and so are the connected stable rank (\cite{CL86b}) and the general stable rank (\cite[Lem.11.1.13]{MR87}). However, the topological stable rank may not be left-right symmetric (\cite{DLMR}). 
\end{rem}
The qualitative similarities between the four stable ranks are displayed in the following table:

\[
\begin{tabular}{rc c}
\multicolumn{1}{r}{}
 &  \multicolumn{1}{c}{topological}
 & \multicolumn{1}{c}{algebraic} \\
dimensional & $\tsr$ & $\bsr$ \\
homotopical & $\csr$ & $\gsr$ \\
\end{tabular}
\]
\smallskip

Quantitatively, the stable ranks are related as follows:

\begin{thm}\label{ord} Let $A$ be a Banach algebra. Then $\gsr A \leq \csr A \leq\bsr A+1\leq\tsr A+1$.
\end{thm}

\begin{thm}\label{HV} Let $A$ be a $\Cstar$-algebra. Then $\bsr A=\tsr A$.
\end{thm}

Theorem~\ref{ord} is due to Rieffel \cite{Rie83}; implicitly, the last three inequalities are also proved by Corach and Larotonda \cite{CL84}. Theorem~\ref{HV}, due to Herman and Vaserstein \cite{HV}, is not true for general Banach algebras (cf. Examples \ref{diskalgebra} and \ref{H-infty}). It would be interesting, however, to extend it beyond the case of $\Cstar$-algebras.

\subsection{Stable rank one}\label{sr1 section} Let us look closely at the distinguished case when stable ranks take on their least possible value. The other extreme case, infinite stable ranks, is discussed in the next paragraph.

\begin{prop}\label{stablerankone} Let $A$ be a Banach algebra. Then the following implications hold:
\begin{eqnarray*}
\xymatrix{
\tsr A=1\;\ar@{=>}[r] &\; \bsr A=1\;\ar@{=>}[d] \\
\csr A=1\;\ar@{=>}[r]  &\; \gsr A=1\;\ar@{=>}[r]  &\; A \textrm{ is stably finite}
}
\end{eqnarray*}
\end{prop}
Recall that a unital algebra $A$ is \emph{finite} if left-invertible implies invertible in $A$, equivalently, if right-invertible implies invertible in $A$; otherwise, $A$ is \emph{infinite}. We say that $A$ is \emph{stably finite} if each matrix algebra $\M_n(A)$ is finite. In what regards the relation between stable rank one and stable finiteness, Proposition~\ref{stablerankone} sharpens and unifies \cite[Prop.3.1]{Rie83} and \cite[Prop.1.15]{Elh95}.

\begin{proof}
By Theorem~\ref{ord}, $\tsr A=1$ implies $\bsr A=1$, and $\csr A=1$ implies $\gsr A=1$. That $\bsr A=1$ implies $\gsr A=1$ is well-known; we include a proof for completeness. If $\bsr A=1$, then $\gsr A\leq 2$ by Theorem~\ref{ord}. In order to have $\gsr A=1$, we need to show that $A^\times$ acts transitively on $\Lg_1(A)$, i.e., that $A$ is finite. Let $a\in A$ be left invertible, say $ba=1$ with $b\in A$. From $ab+(1-ab)=1$ we get $(b, 1-ab)\in\Lg_2(A)$. Then $b+c(1-ab)$ is left invertible for some $c\in A$, since $\bsr A=1$. But $(b+c(1-ab))a=1$, so $a(b+c(1-ab))=1$. Thus $a$ is right-invertible as well, therefore invertible.

For the remaining implication, note that $\gsr A=1$ implies that $A^\times$ acts transitively on $\Lg_1(A)$, in other words $A$ is finite. But $\gsr A=1$ also implies that $\gsr \M_n(A)=1$ for each $n$ (see Corollary~\ref{morita}), so each $\M_n(A)$ is finite. We conclude that $A$ is stably finite.
\end{proof}

In general, the implications in Proposition~\ref{stablerankone} cannot be reversed. Also, having topological or Bass stable rank equal to 1 need not imply that the connected stable rank is 1. In fact, the following proposition - due in part to Elhage Hassan \cite[Prop.1.15]{Elh95} - shows that $\tsr A=1$ implies $\csr A=1$ precisely when $K_1(A)=0$:

\begin{prop}\label{csr vs K1}
Let $A$ be a Banach algebra. Then $\csr A=1$ implies $K_1(A)=0$. Furthermore, if $\tsr A=1$ then the converse holds.
\end{prop}

\begin{proof} If $\csr A=1$ then $\Lg_1(A)=A^\times$ is connected. Since $\csr \M_n(A)=1$ for each $n$ (Corollary~\ref{morita}), we obtain that $\GL_n(A)$ is connected for each $n$. Therefore $K_1(A)=0$. 

For the second part, assume $\tsr A=1$. Then $\csr A\leq 2$ by Theorem~\ref{ord}, and $A$ is finite by Proposition~\ref{stablerankone}. Also, a theorem of Rieffel (\cite[Thm.10.10]{Rie83}, \cite[Thm.2.10]{Rie87}) says that the natural map $\pi_0(A^\times)\to K_1(A)$ is an isomorphism. Hence, if $K_1(A)=0$ then $A^\times=\Lg_1(A)$ is connected, which implies that $\csr A=1$, as desired.
\end{proof}

\begin{rem} A Banach algebra $A$ has $\tsr A=1$ if and only if the invertible elements of $A$ are dense in $A$ (Rieffel \cite[Prop.3.1]{Rie83}).
\end{rem}

\begin{rem}\label{freeness} Inspecting condition ($\gsr '$) of Definition~\ref{definestableranks}, we see that the following are equivalent for a Banach algebra $A$: 

$\cdot$ $\gsr A=1$

$\cdot$ $A$ enjoys the IBN property, and every f.g. stably free $A$-module is free

\noindent Recall that $A$ has the \emph{Invariant Basis Number property} if $A^m\simeq A^n$ (as f.g. right $A$-modules) implies $m=n$; in K-theoretic terms, this is equivalent to $[A]$ having infinite order in $K_0(A)$. Since stable finiteness implies the IBN property, and general stable rank equal to $1$ implies stable finiteness (Proposition~\ref{stablerankone}), we obtain another statement equivalent to $\gsr A=1$:

$\cdot$ $A$ is stably finite, and every f.g. stably free $A$-module is free

\noindent \emph{Complete finiteness} in the sense of Davidson and Ji \cite[Def.2.2]{DJi08} is precisely the property that the general stable rank is equal to 1. The equivalent form given above is an answer to their question of distinguishing the completely finite $\Cstar$-algebras among the stably finite ones.
\end{rem}

\begin{ex}
The irrational rotation $\Cstar$-algebra $A_\theta$ has $\bsr A_\theta=\tsr A_\theta=1$ (Putnam \cite{Put90}). Hence $\gsr A_\theta=1$, and $\csr A_\theta=2$ because $K_1(A_\theta)$ is non-trivial.
\end{ex}

\begin{ex}\label{examplestablerankone}
The reduced $\Cstar$-algebra of a torsion-free, non-elementary hyperbolic group has Bass / topological stable rank $1$ (Dykema and de la Harpe \cite{DdH99}; see also R{\o}rdam \cite{Ror97} for the case of free groups). Thus the general stable rank is $1$, and the connected stable rank is $1$ or $2$ according to whether the $K_1$-group vanishes or not.
\end{ex}

\subsection{Infinite stable rank} The following simple observation is a good source of Banach algebras having all their stable ranks infinite (compare \cite{CL83}):

\begin{prop}\label{lackIBN}
Let $A$ be a Banach algebra. If $[A]$ has finite order in $K_0(A)$, then $\sr A=\infty$. 
\end{prop}

We remind the reader that $\sr A$ denotes any one of $\bsr A$, $\tsr A$, $\csr A$, $\gsr A$. The contrapositive of the above proposition, that $\sr A<\infty$ implies $[A]$ has infinite order in $K_0(A)$, is a relative of the fact that $\sr A=1$ implies stable finiteness of $A$ (Proposition~\ref{stablerankone}).

\begin{proof} It suffices to show that $\gsr A=\infty$. Assume, on the contrary, that $\gsr A$ is finite; we perform the following swindle. Since $A$ does not have the IBN property, we may consider the smallest $n\geq 1$ for which there is some $m>n$ such that $A^n\simeq A^m$ as right $A$-modules. Then $A^n\simeq A^{m+k(m-n)}$ for all $k\geq 0$. Pick $k$ such that $m+k(m-n)\geq \gsr A$. By (gsr$'$) of Definition~\ref{definestableranks} we have $A^{n-1}\simeq A^{m+k(m-n)-1}$, which contradicts the choice of $n$.
\end{proof}

\begin{rem} Elhage Hassan \cite[Prop.1.4]{Elh95} shows that a $\Cstar$-algebra $A$ which contains $n\geq 2$ isometries $s_1, \dots, s_n$ such that $\sum _1 ^n s_is_i^*=1$ has $\csr A=\infty$; actually, the proof shows the stronger fact that $\gsr A=\infty$. We point out that this result can be viewed as a consequence of Proposition~\ref{lackIBN}. Indeed, the $\Cstar$-subalgebra of $A$ generated by $s_1, \dots, s_n$ is isomorphic to the Cuntz algebra $\mathcal{O}_n$. Since $[\mathcal{O}_n]$ has finite order in $K_0(\mathcal{O}_n)$, it follows that $[A]$ has finite order in $K_0(A)$. Hence $\sr A=\infty$. 

From \cite[Prop.1.4]{Elh95}, Elhage Hassan erroneously concludes that every purely infinite, simple $\Cstar$-algebra has infinite connected stable rank. The correct computation appears in the following example. We remind the reader that our $\Cstar$-algebras are assumed to be unital.
\end{rem}

\begin{ex}\label{purelyinfinite} To begin with, consider the general case of an infinite, simple $\Cstar$-algebra: its dimensional stable ranks are infinite (this follows from \cite[Prop.6.5]{Rie83} as soon as one knows \cite[6.11.3]{Bla98}), and the homotopical stable ranks are at least $2$. 

Now let $A$ be a purely infinite, simple $\Cstar$-algebra (\cite{Cun81}; see also \cite[\S 6.11]{Bla98}). The following is due to Xue \cite{Xue99}:
\begin{itemize}
\item[($\Xi$)]  $\csr A$ is  $2$ or $\infty$ according to whether the order of $[A]$ in $K_0(A)$ is infinite or finite.
\end{itemize}
Xue's connected stable rank result has the following general stable rank analogue: 
\begin{itemize}
\item[($\Upsilon$)]  $\gsr A$ is $2$ or $\infty$ according to whether the order of $[A]$ in $K_0(A)$ is infinite or finite.
\end{itemize}
\noindent Despite the similarity between ($\Xi$) and ($\Upsilon$), neither one implies the other: ($\Upsilon$) is half weaker (the $\gsr A=2$ part) and half stronger (the $\gsr A=\infty$ part) than ($\Xi$). While ($\Xi$) is inherently topological, ($\Upsilon$) is actually valid for purely infinite, simple rings; such rings are defined and investigated in \cite{AGP}. In light of Proposition~\ref{lackIBN}, in order to justify ($\Upsilon$) it suffices to argue that $\gsr A=2$ whenever $A$ satisfies the IBN property. In other words, we have to show that, for all $m\geq 2$, if $P$ is a right $A$-module satisfying $P\oplus A\simeq A^m$ then $P\simeq A^{m-1}$. By the IBN property, both $P$ and $A^{m-1}$ are non-zero. Since for non-zero projective f.g. modules, stable isomorphism implies isomorphism - this can be read off from \cite[1.4 \& 1.5]{Cun81} and is stated explicitly in \cite[Prop.2.1]{AGP} - we are done.

To illustrate, consider the Cuntz $\Cstar$-algebras: for $n<\infty$ we have $\csr \mathcal{O}_n=\gsr \mathcal{O}_n=\infty$, whereas $\csr \mathcal{O}_\infty=\gsr \mathcal{O}_\infty=2$. 
\end{ex}

\section{Homotopy invariance}\label{hom} 
Let $A$ and $B$ be Banach algebras. Two morphisms $\phi_0,\phi_1:A\To B$ are \emph{homotopic} if they are the endpoints of a path of morphisms $\{\phi_t\}_{0\leq t\leq 1}:A\To B$; here the continuity of $t\mapsto \phi_t$ is in the pointwise sense, namely $t\mapsto\phi_t(a)$ is continuous for each $a\in A$. If there are morphisms $\phi:A\To B$ and $\psi:B\To A$ with $\psi\phi$ homotopic to $\id_A$ and $\phi\psi$ homotopic to $\id_B$, then $A$ and $B$ are said to be \emph{homotopy equivalent}. This notion generalizes the usual homotopy equivalence: two compact spaces $X$ and $Y$ are homotopy equivalent (as spaces) if and only if $C(X)$ and $C(Y)$ are homotopy equivalent (as Banach algebras).

It is clear that the topological and the Bass stable ranks are not homotopy invariant. On the other hand, the connected and the general stable ranks are homotopy invariant. For the connected stable rank, this is due to Nistor \cite[Lem.2.8]{Nis86}.

\begin{thm}\label{csrgsrhomotopy} If $A$ and $B$ are homotopy equivalent Banach algebras, then $\csr A=\csr B$ and $\gsr A=\gsr B$.
\end{thm}

\begin{proof} We claim that $\csr A\leq\csr B$ and $\gsr A\leq\gsr B$ provided $\phi:A\To B$ and $\psi:B\To A$ are morphisms such that $\psi\phi$ homotopic to $\id_A$. First, note the following: if $\underline{a}\in \Lg_m(A)$, then $\psi\phi(\underline{a})$ is in the same component as $\underline{a}$; in other words, $\psi\phi(\underline{a})$ and $\underline{a}$ are in the same $\GL^0_m(A)$-orbit.

We show that $\csr A\leq\csr B$. Let $m\geq\csr B$, and pick $\underline{a}$ and $\underline{a'}$ in $\Lg_m(A)$. Then $\phi(\underline{a})$ and $\phi(\underline{a'})$ are in $\Lg_m(B)$, so there is $\beta\in\GL^0_m(B)$ taking $\phi(\underline{a})$ to $\phi(\underline{a'})$. Sending this through $\psi$, we get $\psi(\beta)\in\GL^0_m(A)$ taking $\psi\phi(\underline{a})$ to $\psi\phi(\underline{a'})$. Hence $\underline{a}$ and $\underline{a'}$ are in the same $\GL^0_m(A)$-orbit.

We show that $\gsr A\leq\gsr B$. Let $m\geq\gsr B$, and pick $\underline{a}$ and $\underline{a'}$ in $\Lg_m(A)$. The argument runs just like the one for $\csr$, except that we get some $\beta$ in $\GL_m(B)$, rather than in $\GL^0_m(B)$, taking $\phi(\underline{a})$ to $\phi(\underline{a'})$. The conclusion is that $\underline{a}$ and $\underline{a'}$ are in the same $\GL_m(A)$-orbit.
\end{proof}

\begin{cor}\label{homtype}
The connected and the general stable ranks of $C(X)$ only depend on the homotopy type of the compact space $X$. In particular, $\csr C(X)=\gsr C(X)=1$ if $X$ is contractible. 
\end{cor}


\section{Commutative $\Cstar$-algebras}\label{SectionC(X)}
For a compact space $X$, the topological and the Bass stable ranks of $C(X)$ can be computed in terms of the (covering) dimension of $X$. As in manifold theory, we use the notation $X^d$ to indicate that $X$ is $d$-dimensional.

The following is  due to Vaserstein \cite[Thm.7]{Vas71} for the Bass stable rank, and to Rieffel \cite[Prop.1.7]{Rie83} for the topological stable rank:

\begin{thm}\label{bsrtsr from dimension} Let $X^d$ be a compact space. Then $\bsr C(X^d)=\tsr C(X^d)=\lfloor d/2 \rfloor+1$.
\end{thm}

Consequently, $\csr C(X^d)\leq \lfloor d/2 \rfloor+2$ by Theorem~\ref{ord}. A better estimate, obtained by Nistor \cite[Cor.2.5]{Nis86}, is the following:

\begin{thm}\label{csrbound} Let $X^d$ be a compact space. Then $\csr C(X^d)\leq\lceil d/2 \rceil+1$.
\end{thm}

Since the connected stable rank of $C(X)$ only depends on the homotopy type of $X$, it is clear that the dimensional upper bound in the previous theorem is not necessarily attained. The following criterion appears in \cite[Prop.28]{Nic08}, with a self-contained proof, in the case of finite-dimensional CW-complexes. As indicated there, the result is true for compact metric spaces.

\begin{thm}\label{csrdim} Let $X^d$ be a compact metric space of finite dimension.

a) If $d$ is odd, then $\csr C(X^d)=\lceil d/2 \rceil+1$ if and only if $H^d(X^d)\neq 0$.

b) If $d$ is even, then $\csr C(X^d)=\lceil d/2 \rceil+1$ provided $H^{d-1}(X^d)\neq 0$.
\end{thm}
Roughly speaking, this theorem says that the connected stable rank of $C(X^d)$ attains its dimensional upper bound as soon as the top cohomology group in $H^{\textrm{odd}}(X^d)$ is non-vanishing. Let us point out that the cohomology is taken in the \v{C}ech sense, and with integer coefficients.

\begin{proof}
$\Lg_m C(X^d)$ can be identified with the space of continuous maps from $X^d$ to $\C^m\setminus \{0\}$. Hence, $\Lg_m C(X^d)$ is connected if and only if  the set of homotopy classes $[X^d,\C^m\setminus \{0\}]=[X^d,S^{2m-1}]$ degenerates to a singleton. It follows that $\csr C(X^d)$ is the least $n\geq 1$ such that $[X^d,S^{2m-1}]$ degenerates for all $m\geq n$. This reformulation has several consequences. First, we see once again that $\csr C(X^d)$ only depends on the homotopy type of $X^d$. Second, we get a direct and conceptual proof of Theorem~\ref{csrbound}: we have $\csr C(X^d)\leq\lceil d/2 \rceil+1$ because, for $m\geq \lceil d/2 \rceil+1$, we get $2m-1> d$ hence $[X^d,S^{2m-1}]$ degenerates (\cite{HW}, Theorem VI.6 on p.88). Finally, we have $\csr C(X^d)=\lceil d/2 \rceil+1$ if and only if $[X^d,S^{2m-1}]$ is non-degenerate for $m=\lceil d/2 \rceil$. According to whether $d$ is odd or even, the latter condition amounts to $[X^d,S^d]$, respectively $[X^d,S^{d-1}]$, being non-degenerate. The proof is completed as soon as we recall that $H^d(X^d)\neq 0$ if and only if $[X^d, S^d]$ is non-degenerate (\cite{HW}, corollary on p.150), and that  $H^{d-1}(X^d)\neq 0$ implies $[X^d, S^{d-1}]$ is non-degenerate (\cite{HW}, Corollary 1 on p.149). \end{proof}

Theorem~\ref{csrdim} applies to many familiar spaces (e.g., the tori $T^d$). However, Theorem~\ref{csrdim} is not exhaustive: for instance, it does not apply directly to even-dimensional spheres. To cover this case, we revisit the proof of Theorem~\ref{csrdim} at the point where homotopy was still involved. For even $d$, we have seen that $\csr C(S^d)=\lceil d/2 \rceil+1$ if and only if $[S^d,S^{d-1}]$ is non-degenerate. Since $\pi_d(S^{d-1})$ vanishes for $d=2$ only, we conclude:

\begin{ex}\label{csrspheres}
\begin{displaymath}
\csr C(S^d)=\left\{\begin{array}{ll}
\lceil d/2 \rceil+1 &  \textrm{ if $d\neq 2$}\\
1 &  \textrm{  if $d=2$}
\end{array}
\right.
\end{displaymath}
\end{ex}

The computation of the general stable rank of $C(X)$ is much more complicated than the computation of its connected stable rank. There are properties of a compact space $X$ - contractibility (Corollary~\ref{homtype}), or low dimensionality (Proposition~\ref{csrlowdim}) - which guarantee that $\gsr C(X)=1$. Other computations of general stable ranks, particularly those yielding higher values, are harder to provide. The following result is the first non-trivial computation of this kind:

\begin{prop}\label{gsrspheres} The general stable rank of $C(S^d)$ is given as follows:
\begin{displaymath}
\gsr C(S^d)=\left\{\begin{array}{ll}
\lceil d/2 \rceil+1 &  \textrm{ if $d>4$ and $d\notin 4\Z$}\\
\lceil d/2 \rceil &  \textrm{ if $d>4$ and $d\in 4\Z$}\\
1 &  \textrm{ if $d\leq 4$}
\end{array}
\right.
\end{displaymath}
\end{prop}

\begin{proof}
Let $X$ be a compact space. By (gsr$'$) of Definition~\ref{definestableranks}, $\gsr C(X)$ is the least $n\geq 1$ with the following property: for all $m\geq n$, if $P$ is a right $C(X)$-module satisfying $P\oplus C(X)\simeq C(X)^m$ then $P\simeq C(X)^{m-1}$. Via the Serre-Swan dictionary, we can translate this algebraic description into a geometric one involving complex vector bundles. Namely, $\gsr C(X)$ is the least integer $n$ with the following property: for all $m\geq n$, if $E$ is an $(m-1)$-dimensional vector bundle over $X$ which is trivialized by adding a $1$-dimensional vector bundle over $X$, then $E$ is trivial. 

Recall that there is a bijective correspondence 
\[\big[\mathrm{Vect}_n(SX)\big]\longleftrightarrow \big[X,\GL_n(\C)\big]\] 
between the isomorphism classes of $n$-dimensional complex vector bundles over the suspension $SX$ and the homotopy classes of continuous maps $X\to \GL_n(\C)$ (see \cite[p.36]{Kar78}). This correspondence is implemented by clutching. View $SX$ as the union of two cones over $X$, denoted $X_+$ and $X_-$. On $X_+$, respectively $X_-$, take the trivial vector bundle $X_+\times \C^n$, respectively $X_-\times \C^n$. We glue these trivial bundles along $X=X_+\cap X_-$ by a continuous map $f:X\To \GL_n(\C)$; specifically, the two copies of $\C^n$ above each $x\in X$ get identified by the linear isomorphism $f(x)$. We thus have an $n$-dimensional vector bundle $E_f$ over $SX$ for each continuous map $f:X\To \GL_n(\C)$. Up to isomorphism,  each $n$-dimensional vector bundle over $SX$ arises in this way. Indeed, as $X_+$ and $X_-$ are contractible, every vector bundle over $X$ restricts to trivial vector bundles over $X_+$ and $X_-$; thus all that matters is the way these two trivial vector bundles fit together over $X$. 

Furthermore, the direct sum of bundles obtained by clutching behaves as expected (\cite[p.136]{Hus94}): 
\[E_f\oplus E_g\simeq E_{\bigl(\begin{smallmatrix} f & 0\\0 & g\end{smallmatrix}\bigr)}\]
Therefore, if $E_f$ is the $n$-dimensional bundle determined by $f: X\to\GL_n(\C)$, then $E_f$ is trivial if and only if $f$ vanishes in $\big[X,\GL_n(\C)\big]$, and $E_f$ is stably trivial if and only if $f$ vanishes in $\big[X,\GL_m(\C)\big]$ for some $m\geq n$. To put it differently, there is a bijective correspondence between non-zero elements in the kernel of $\big[X,\GL_n(\C)\big]\to \big[X,\GL_{n+1}(\C)\big]$, and non-trivial $n$-dimensional vector bundles which become trivial after adding a $1$-dimensional vector bundle.

To summarize, $\gsr C(SX)$ is the least $n\geq 1$ such that $\big[X,\GL_{m-1}(\C)\big]\to \big[X,\GL_m(\C)\big]$ is injective for all $m\geq n$.

Now we let $X=S^*$ be a sphere, and we recall that the unitary group $\U(n)$ is a deformation retract of $\GL_n(\C)$. Then $\gsr C(S^{*+1})$ is the least $n\geq 1$ for which $\pi_*\U(m-1)\to \pi_*\U(m)$ is injective for all $m\geq n$. Let us also recall at this point that the long exact homotopy sequence associated to the fibration $\U(n)\to \U(n+1)\to S^{2n+1}$ yields that $\pi_*\U(n)\to \pi_*\U(n+1)$ is bijective for $n>*/2$. 

When $*\leq 3$, one easily checks that  $\pi_*\U(m-1)\to \pi_*\U(m)$ is injective for all $m\geq 1$. Therefore $\gsr C(S^d)=1$ for $d\leq 4$ (compare Proposition~\ref{csrlowdim} below).

Assume $*\geq 4$. In order to see what happens right before the stable range $n>*/2$, we use some computations of homotopy groups of unitary groups as tabulated in \cite[p.254]{Lun}. We split the analysis according to the parity of $*$:

(even $*$) Put $*=2k$ with $k\geq 2$. The sequence of homotopy groups $\{\pi_{2k}\U(n)\}_{n\geq 1}$ stabilize starting from $\pi_{2k}\U(k+1)$, and $\pi_{2k}\U(k+1)\simeq\pi_{2k}\U(\infty)\simeq 0$ by Bott periodicity. The last unstable group is $\pi_{2k}\U(k)\simeq\Z_{k!}$, so the map $\pi_{2k}\U(k)\to \pi_{2k}\U(k+1)$ is not injective. Thus $\gsr C(S^{2k+1})=k+2$ for $k\geq 2$.

(odd $*$) Put $*=2k+1$ with $k\geq 2$. The sequence of homotopy groups $\{\pi_{2k+1}\U(n)\}_{n\geq 1}$ stabilize starting from $\pi_{2k+1}\U(k+1)$, and $\pi_{2k+1}\U(k+1)\simeq\pi_{2k+1}\U(\infty)\simeq \Z$ by Bott periodicity. The last unstable group is
\begin{displaymath}
\pi_{2k+1}\U(k)\simeq \left\{\begin{array}{ll}
\Z_2 &  \textrm{ if $k$ is even,}\\
0 &  \textrm{  if $k$ is odd.}
\end{array}
\right.
\end{displaymath}
If $k$ is even, then $\pi_{2k+1}\U(k)\to \pi_{2k+1}\U(k+1)$ is not injective. Therefore $\gsr C(S^{2k+2})=k+2$ for even $k\geq 2$. If $k$ is odd, we must look at the map $\pi_{2k+1}\U(k-1)\to \pi_{2k+1}\U(k)$ in order to see the failure of injectivity: indeed, $\pi_{2k+1}\U(k-1)$ has a cyclic group of order $\gcd(k-1,8)$ (which is not 1, since $k$ is odd) as a direct summand. Thus $\gsr C(S^{2k+2})=k+1$ for odd $k\geq 2$.
\end{proof}

\begin{rem} A related, but incomplete, discussion along these lines is carried out by Sheu in \cite[pp.369--370]{She87}; in particular, the previous proposition confirms his conjectural Remark on p.370. We also point out that the homotopical information used in the proof of Proposition~\ref{gsrspheres} quickly leads to an exact computation of the constant $C_G$ which appears in the main theorem (5.4) of \cite{She87}.
\end{rem}

\begin{prop}\label{csrlowdim} Let $X^d$ be a compact space. If $d\leq 4$, then $\gsr C(X^d)=1$.
\end{prop}

\begin{proof} If $d\leq 4$ then $\csr C(X^d)\leq 3$, by Theorem~\ref{csrbound}. Therefore $\GL_n(C(X^d))$ acts transitively on $\Lg_n(C(X^d))$ for $n\geq 3$. That $\GL_n(C(X^d))$ acts transitively on $\Lg_n(C(X^d))$ for $n=1,2$ is automatic by the commutativity of $C(X^d)$. We conclude that $\gsr C(X^d)=1$. \end{proof}

Sadly, we do not know the answer to the following:

\begin{prob}\label{gsrtori}
Compute $\gsr C(T^d)$.
\end{prob}

To get a sense of why this problem is much more challenging than the computation of $\gsr C(S^d)$, one need only glance at the proof of Packer and Rieffel \cite{PR03} that $\gsr C(T^5)>1$.

\section{The Gelfand transform} \label{Gelfand section}
We remain in the commutative case, and we consider the transfer of stable rank information across the Gelfand transform. This discussion owes much to Taylor's papers \cite{Tay75}, \cite{Tay76}.

For a unital commutative Banach algebra $A$, the \emph{maximal ideal space} $X_A$ is the set of characters of $A$. Equipped with the topology of pointwise convergence, $X_A$ is a compact Hausdorff space. The \emph{Gelfand transform} of $A$ is the unital, continuous morphism $A\To C(X_A)$ given by $a\mapsto \hat{a}$, where $\hat{a}\in C(X_A)$ denotes the evaluation at $a\in A$. The fundamental feature of the Gelfand transform is the fact that it is spectrum-preserving: $\spec_{C(X_A)} (\hat{a})=\spec_A (a)$ for all $a\in A$. 

Strong relations between the structural properties of $A$ and those of $C(X_A)$ can be established across the Gelfand transform. Early results of Shilov and Arens led Novodvorskii \cite{Nov67} to the following important theorem: the Gelfand transform $A\To C(X_A)$ induces an isomorphism $K_*(A)\to K_*(C(X_A))$. 

For the homotopical stable ranks, we have:

\begin{thm}\label{Gelfand csrgsr}
Let $A\To C(X_A)$ be the Gelfand transform. Then $\csr A= \csr C(X_A)$ and $\gsr A=\gsr C(X_A)$.
\end{thm}

The equality of connected stable ranks in the above theorem is a consequence of Novodvorskii's results from \cite{Nov67}, which imply that the Gelfand transform $A\To C(X_A)$ induces, for all $n\geq 1$, a bijection $\pi_0(\Lg_n(A))\to \pi_0(\Lg_n(C(X_A)))$. For the equality of general stable ranks, one can appeal to the following result of Forster \cite[Thm.6]{For74} and Taylor \cite[Thm.6.8]{Tay76}: the Gelfand transform $A\To C(X_A)$ induces a monoid isomorphism $\mathcal{P}(A)\to\mathcal{P}(C(X_A))$. By (gsr$'$) of Definition~\ref{definestableranks}, where the general stable rank is defined in terms of a cancellation property for projective modules, we immediately obtain the desired equality of general stable ranks.

Passing to the dimensional stable ranks, we have the following theorem of Corach and Larotonda \cite[Thm.8]{CL84}:
\begin{thm}\label{Gelfand bsr}
Let $A\To C(X_A)$ be the Gelfand transform. Then $\bsr A\leq \bsr C(X_A)$.
\end{thm}

In general, the inequality of Bass stable ranks can be strict, see Examples \ref{diskalgebra} and \ref{H-infty} below.

As for the topological stable rank, no general comparison between $\tsr A$ and $\tsr C(X_A)$ seems to be known. Let us record this question, which was first raised in \cite[\S3]{CS88}:

\begin{prob}\label{Gelfand tsr}
Let $A\To C(X_A)$ be the Gelfand transform. Does $\tsr A\leq \tsr C(X_A)$ hold? Does $\tsr A\geq \tsr C(X_A)$ hold?
\end{prob}

We now look at some examples.

\begin{ex}\label{Z^d} The maximal ideal space of $\ell^1(\Z^d)$ is homeomorphic to the $d$-torus $T^d$ and, with this identification, the Gelfand transform is the inclusion $\ell^1(\Z^d)\into C(T^d)$. Theorem~\ref{Gelfand csrgsr} yields $\csr \ell^1(\Z^d)=\csr C(T^d)(=\lceil d/2 \rceil+1)$ and $\gsr \ell^1(\Z^d)= \gsr C(T^d)$. Pannenberg \cite[Cor.4]{Pan90} (see also Mikkola and Sasane \cite[Cor.5.14]{MS09}) computed $\tsr \ell^1(\Z^d)=\lfloor d/2 \rfloor+1$; consequently, $\tsr \ell^1(\Z^d)=\tsr C(T^d)$. For the Bass stable rank, Theorem~\ref{Gelfand bsr} together with Remark~\ref{bsrquestion by Vaserstein} give $\bsr \ell^1(\Z^d)= \bsr C(T^d)$. Summarizing, we have $\sr \ell^1(\Z^d)= \sr C(T^d)$.
\end{ex}

In the next two examples, we consider two prominent Banach algebras of holomorphic functions on the open unit disk $D=\{z\in \C: |z|<1\}$.

\begin{ex}\label{diskalgebra} 
The \emph{disk algebra} $A(D)$ is the closed subalgebra of $C(\overline{D})$ consisting of those functions which are holomorphic on $D$. The maximal ideal space $X_{A(D)}$ is homeomorphic to $\overline{D}$ and, with this identification, the Gelfand transform is the inclusion $A(D)\into C(\overline{D})$.

The dimensional stable ranks are $\bsr A(D)=1$ (Corach and Su\'{a}rez \cite{CS85}, Jones, Marshall and Wolff \cite{JMW86}) and $\tsr A(D)=2$. In particular, $\bsr A(D)<\bsr C(\overline{D})$ and  $\tsr A(D)=\tsr C(\overline{D})$. 

The homotopical stable ranks are $\csr A(D)=\gsr A(D)=1$. Indeed, the contractibility of $\overline{D}$ yields $\csr C(\overline{D})=\gsr C(\overline{D})=1$; now apply Theorem \ref{Gelfand csrgsr}.
\end{ex}

\begin{ex}\label{H-infty} The \emph{Hardy algebra} $H^\infty(D)$ is the algebra of bounded holomorphic functions on $D$. This time, no concrete topological description for the maximal ideal space $X_{H^\infty(D)}$ is available. It is known, however, that $\dim X_{H^\infty(D)}=2$ (Su\'{a}rez \cite{Sua94}).

The dimensional stable ranks are $\bsr H^\infty(D)=1$ (Treil \cite{Tre92}) and $\tsr H^\infty(D) =2$ (Su\'{a}rez \cite{Sua96}). We have again $\bsr H^\infty(D)<\bsr C(X_{H^\infty(D)})$ and  $\tsr H^\infty(D)=\tsr C(X_{H^\infty(D)})$. 

The homotopical stable ranks are $\csr H^\infty(D)=2$ and $\gsr H^\infty(D)=1$. The latter follows from $\bsr H^\infty(D)=1$. The connected stable rank formula follows from $\csr H^\infty(D)\leq 2$, together with the fact that the invertible group of $H^\infty(D)$ is not connected (see the introduction of \cite{Sua94}). 
\end{ex}

\begin{rem} In \cite{BS09}, Brudnyi and Sasane investigate \emph{projective-free} commutative Banach algebras, i.e., commutative Banach algebras with the property that every f.g. projective module is free. We point out that the Forster - Taylor isomorphism $\mathcal{P}(A)\to\mathcal{P}(C(X_A))$, mentioned above, implies one of the main results of Brudnyi and Sasane: a commutative Banach algebra $A$ is projective-free if and only if $C(X_A)$ is projective-free (compare \cite[Thm.1.2]{BS09}).
\end{rem}


\section{Matrix algebras} 
The following was proved by Vaserstein \cite[Thm.3]{Vas71} for the Bass stable rank, and by Rieffel \cite[Thm.6.1]{Rie83} for the topological stable rank:

\begin{thm}\label{matrixtsrbsr} Let $A$ be a Banach algebra. Then:
\[\tsr \M_n(A) =\Big\lceil\tfrac{1}{n}(\tsr A -1)\Big\rceil+1,\qquad \bsr \M_n(A) =\Big\lceil\tfrac{1}{n}(\bsr A -1)\Big\rceil+1 \]
\end{thm}

For the homotopical stable ranks we have:

\begin{thm}\label{matrixcsrgsr} Let $A$ be a Banach algebra. Then:
\[\csr \M_n(A) \leq\Big\lceil\tfrac{1}{n}(\csr A-1)\Big\rceil+1,\quad \gsr\M_n(A) \leq\Big\lceil\tfrac{1}{n}(\gsr A-1)\Big\rceil+1\]
\end{thm}

Both Nistor \cite[Prop.2.10]{Nis86} and Rieffel \cite[Thm.4.7]{Rie87} showed the connected stable rank estimate. The one concerning the general stable rank appears in \cite[Cor.11.5.13]{MR87} (note that gsr as defined in \cite{MR87} equals $\gsr-1$ as defined here).

An important consequence is the fact that having stable rank equal to $1$ is a stable property:

\begin{cor}\label{morita} If $\sr A =1$ then $\sr \M_n(A)=1$.
\end{cor}

\begin{rem} The inequalities in Theorem~\ref{matrixcsrgsr} can be strict, as the following example shows. Let $A=C(S^{2d})$ with $d\geq 3$, and let $n>d$. Then, in both inequalities, the right-hand side equals $2$ (recall Example~\ref{csrspheres} and Proposition~\ref{gsrspheres}) whereas the left-hand side equals $1$ ($\csr\M_n(A)=1$ follows from $\csr\M_n(A)\leq 2$ and the vanishing of $\pi_0(\GL_n(A))\simeq \pi_{2d} \U(n)$; $\gsr\M_n(A)=1$ follows from $\gsr\M_n(A)\leq 2$ and the finiteness of $\M_n(A)$).
\end{rem}

\section{Quotients}\label{Sectiononto}
The following result is due to Vaserstein \cite[Thm.7]{Vas71} for the Bass stable rank, and to Rieffel \cite[Thm.4.3]{Rie83} for the topological stable rank:

\begin{thm}\label{ontotsrbsr} Let $\pi: A\To B$ be an epimorphism. Then $\tsr B \leq \tsr A$ and $\bsr B \leq \bsr A$.
\end{thm}

We now consider the homotopical stable ranks. An example as simple as $C(I^d)\onto C(\partial I^d)$ shows that we cannot expect to have $\csr B\leq \csr A$, or $\gsr B\leq \gsr A$, whenever $B$ is a quotient of $A$. However, we have the following:

\begin{thm}\label{ontocsrgsr} Let $\pi: A\To B$ be an epimorphism. Then:
\[ \csr B \leq \max\{\csr A ,\bsr A\}, \qquad\gsr B \leq \max\{\gsr A ,\bsr A \}\]
\end{thm}

The connected stable rank estimate from Theorem~\ref{ontocsrgsr} is due to Elhage Hassan \cite[Thm.1.1]{Elh95}.

\begin{proof} 
We use the following fact (\cite[Lem.4.1]{Bas64}): 
\begin{itemize}
\item[($\dagger$)] If $\pi:A\To B$ is onto, then $\pi: \Lg_n(A)\To \Lg_n(B)$ is onto for $n\geq\bsr A$.
\end{itemize}
Indeed, let $(b_i)\in \Lg_n(B)$. Pick $(b'_i)\in B^n$ such that $\sum b_i'b_i=1$, and $(a_i), (a'_i)\in A^n$ with $\pi(a_i)=b_i$,  $\pi(a_i')=b_i'$. Then $\pi(\sum a_i'a_i)=1$, that is, $\sum a_i'a_i=1+k$ for some $k\in\ker\pi$. As $(a_1,\dots,a_n,k)\in \Lg_{n+1}(A)$ and $n\geq\bsr A$, we get $(a_i+c_ik)\in \Lg_n(A)$ for some $(c_i)\in A^n$. We are done, since $\pi(a_i+c_ik)=(b_i)$.

Let $m\geq \max\{\csr A,\bsr A\}$. Then $\Lg_m(A)$ is connected, and $\Lg_m(A)$ maps onto $\Lg_m(B)$, by $(\dagger)$. It follows that $\Lg_m(B)$ is connected, so $\csr B \leq \max\{\csr A ,\bsr A \}$. This is the proof given in \cite[Thm.1.1]{Elh95}. Let us give another argument. We claim that $\GL^0_m(B)$ acts transitively on $\Lg_m(B)$. Let $\underline{b}, \underline{b'}\in \Lg_m(B)$. By $(\dagger)$, we can pick $\underline{a}, \underline{a'}\in\Lg_m(A)$ which are $\pi$-lifts of $\underline{b}, \underline{b'}$. There is $\alpha\in\GL^0_m(A)$ so that $\alpha\cdot \underline{a}^T=\underline{a'}^T$; hence $\pi(\alpha)\cdot \underline{b}^T=\underline{b'}^T$ with $\pi(\alpha)\in\GL^0_m(B)$. We conclude that $\csr B \leq \max\{\csr A ,\bsr A\}$.

It is obvious how to adapt the second proof so as to handle the general stable rank estimate. Let $m\geq \max\{\gsr A,\bsr A\}$; we want to show that $\GL_m(B)$ acts transitively on $\Lg_m(B)$. Let $\underline{b}, \underline{b'}\in \Lg_m(B)$, and let $\underline{a}, \underline{a'}\in\Lg_m(A)$ be $\pi$-lifts of $\underline{b}, \underline{b'}$. There is $\alpha\in\GL_m(A)$ taking $\underline{a}$ to $\underline{a'}$; hence $\pi(\alpha)\in\GL_m(B)$ takes $\underline{b}$ to $\underline{b'}$.
\end{proof}

An epimorphism $\pi:A\to B$ is said to be  \emph{split} if there is a section morphism $s: B\to A$ such that $\pi\circ s=\id_B$. For such epimorphisms, the proof of Theorem~\ref{csrgsrhomotopy} shows the following:

\begin{prop}\label{ontosplit} Let $\pi: A\To B$ be a split epimorphism. Then $\csr B \leq \csr A$ and $\gsr B \leq \gsr A$.
\end{prop}


\section{Dense morphisms}\label{SectionDense}
A morphism $\phi:A\To B$ between Banach algebras is \emph{dense} if $\phi(A)$ is dense in $B$. Note that dense morphisms are automatically onto in the $\Cstar$-context. However, dense morphisms which are not surjective appear naturally when we work in the Banach (or the Fr\'echet) category.

For the topological stable rank we have the following observation (cf. \cite[Prop.4.12]{Bad98}):

\begin{thm}\label{densetsr} Let $\phi: A\To B$ be a dense morphism. Then $\tsr B\leq \tsr A$.
\end{thm}

\begin{proof} Put $n=\tsr A$. Then $\Lg_n(A)$ is dense in $A^n$, hence $\phi(\Lg_n(A))$ is dense in $\phi(A^n)$. But $\phi(A^n)$ is dense in $B^n$, so $\phi(\Lg_n(A))$ is dense in $B^n$. Therefore $\Lg_n(B)$ is dense in $B^n$.
\end{proof}

\begin{ex} 
Consider the dense inclusion of $\ell^1F_n$ into the full group $\Cstar$-algebra $\Cstar F_n$, where $F_n$ is the free group on $n\geq 2$ generators.
As shown by Joel Anderson in \cite[Thm.6.7]{Rie83}, $\tsr \Cstar F_n=\infty$; hence $\tsr \ell^1F_n=\infty$. On the other hand, $\tsr \Cred F_n=1$ (Example~\ref{examplestablerankone}).
\end{ex}

Comparing Theorem~\ref{densetsr} and Theorem~\ref{ontotsrbsr}, we are led to the following: 

\begin{prob}\label{bsrquestion}
Let $\phi: A\To B$ be a dense morphism. Is $\bsr B\leq\bsr A$?
\end{prob}

\begin{rem}\label{bsrquestion by Vaserstein} A theorem of Vaserstein \cite[Thm.7]{Vas71} gives a partial answer to the above problem: if $A$ is a dense subalgebra of $C(X)$, where $X$ is a compact space, then $\bsr C(X)\leq \bsr A$. 
\end{rem}

The next result should be compared with Theorem~\ref{ontocsrgsr}:

\begin{thm}\label{densecsrgsr} Let $\phi: A\To B$ be a dense morphism. Then: 
\[\csr B  \leq \max\{\csr A ,\tsr A\},\qquad \gsr B  \leq \max\{\gsr A ,\tsr A\}\]
\end{thm}

\begin{proof} 
We show that $\csr B\leq \max\{\csr A ,\tsr A\}$. Let $m\geq\max\{\csr A,\tsr A\}$. We have seen in the previous proof that $\phi(\Lg_m(A))$ is dense in $\Lg_m(B)$ for $m\geq\tsr A$. Since $m\geq\csr A$, $\Lg_m(A)$ is connected and so $\phi( \Lg_m(A))$ is connected. It follows that $\Lg_m(B)$ is connected, as it contains a dense connected subset.

Let us give another argument. We show that the action of $\GL^0_m(B)$ on $\Lg_m(B)$ is transitive. Let $\underline{b}\in \Lg_m(B)$. Due to the density of $\phi(\Lg_m(A))$ in $\Lg_m(B)$, we may pick $\underline{a}\in \Lg_m(A)$ such that $\phi(\underline{a})$ is in the $\GL^0_m(B)$-orbit of $\underline{b}$. Since $m\geq \csr A$, there is $\alpha\in \GL^0_m(A)$ taking $(0,\dots, 0,1)\in A^m$ to $\underline{a}$. Then $\phi(\alpha)\in \GL^0_m(B)$ takes $(0,\dots, 0,1)\in B^m$ to $\phi(\underline{a})$. Therefore $\underline{b}\in \Lg_m(B)$ is in the $\GL^0_m(B)$-orbit of $(0,\dots, 0,1)\in B^m$.

Although slightly longer, the second argument has the advantage of being easily adaptable so as to yield the general stable rank estimate. To spell it out, we claim that the action of $\GL_m(B)$ on $\Lg_m(B)$ is transitive whenever $m\geq\max\{\gsr A,\tsr A\}$. Let $\underline{b}\in \Lg_m(B)$, and pick $\underline{a}\in \Lg_m(A)$ such that $\phi(\underline{a})$ is in the $\GL^0_m(B)$-orbit of $\underline{b}$; as before, we use here the density of $\phi(\Lg_m(A))$ in $\Lg_m(B)$ - available as soon as $m\geq\tsr A$. Since $m\geq \gsr A$, there is $\alpha\in \GL_m(A)$ taking the last basis vector $(0,\dots, 0,1)\in A^m$ to $\underline{a}$. Then $\phi(\alpha)\in \GL_m(B)$ takes $(0,\dots, 0,1)\in B^m$ to $\phi(\underline{a})$. Therefore $\underline{b}\in \Lg_m(B)$ is in the $\GL_m(B)$-orbit of $(0,\dots, 0,1)\in B^m$, as desired. 
\end{proof}

\section{Swan's problem}\label{SectionSwan}
A Banach algebra morphism $\phi: A\to B$ is said to be \emph{spectral} if it is spectrum-preserving, that is, $\spec_B (\phi(a))=\spec_A (a)$ for all $a\in A$. Equivalently, the morphism $\phi$ is spectral if, for all $a\in A$, we have that $a$ is invertible in $A$ if and only if $\phi(a)$ is invertible in $B$. 

The Gelfand transform is an example of spectral morphism. In Section~\ref{Gelfand section}, we compared stable ranks across the Gelfand transform, and we saw that the homotopical stable ranks are better behaved than the dimensional stable ranks. In this section, we give up the commutative context of Section~\ref{Gelfand section}. Instead, the spectral morphisms we consider are assumed to be dense, i.e., they have dense image. (The Gelfand transform may or may not be dense.) Following the theme of the paper, we are interested in the following problem raised by Swan \cite[p.206]{Swa77}: how are stable ranks related across a dense and spectral morphism? In \cite{Swa77}, Swan was working with the Bass stable rank and a certain projective stable rank; however, the above problem has since been considered for other stable ranks, as well (see, for instance, \cite{Bad98}).

Let us give some examples of dense and spectral morphisms. We start with a commutative one: if $M$ is a compact manifold, then the inclusion $C^k(M)\into C(M)$ is dense and spectral. Here $C^k(M)$ is a Banach algebra under the norm $\|f\|_{(k)}:= \sum _{|\alpha|\leq k} \|\partial^{\alpha}f\|_\infty$, defined using local charts on $M$. A metric cousin of this example is the following: if $X$ is a compact metric space, then the inclusion $\mathrm{Lip}(X)\into C(X)$ is dense and spectral. By $\mathrm{Lip}(X)$ we denote the Banach algebra of Lipschitz functions on $X$, normed by $\|\cdot\|_\infty+\|\cdot\|_\mathrm{Lip}$, and we think of it as an ersatz $C^1(X)$. In fact, in the spirit of Noncommutative Geometry (Connes \cite{Con94}), one turns these examples into the idea that a dense and spectral Banach subalgebra  of a $\Cstar$-algebra is a ``smooth'' subalgebra carrying ``differential'' information about the ``space''. 

It may not be apparent from the definition, but this idea underlies our next example of dense and spectral morphism. Let $\G$ be a finitely generated group, equipped with a word-length $|\cdot|$. Following Jolissaint \cite{Jol90}, we define the $s$-Sobolev space $H^s\G$ as the completion of $\C\G$ under the weighted $\ell^2$-norm $\big\|\sum a_gg\big\|_{2,s}:=\big(\sum |a_g|^2(1+|g|)^{2s}\big)^{1/2}$.
The group $\G$ is said to have \emph{property RD (of order $s$)} if there are constants $C,s\geq 0$ such that $\|a\|\leq C\|a\|_{2,s}$ for all $a\in \C\G$, where $\|\cdot\|$ denotes the operator norm coming from the regular representation of $\G$ on $\ell^2\G$. Implicitly, this property first appeared in Haagerup's influential paper \cite{Haa79} in the case of free groups. The explicit definition is due to Jolissaint \cite{Jol90}, who proved - among other things - that groups of polynomial growth have property RD. Many more groups are known to satisfy property RD, e.g., all hyperbolic groups (de la Harpe \cite{dHa88}). Now, the relevant fact about property RD is the following: if $\G$ has property RD of order $s$, then for every $S>s$ the $S$-Sobolev space $H^S\G$ is a Banach algebra under $\|\cdot\|_{2,S}$, and the continuous inclusion $H^S\G\into \Cred\G$ is dense and spectral (Lafforgue \cite[Prop.1.2]{Laf00}).

The last example we mention is the result of Ludwig \cite{Lud79} saying that, for a finitely generated group $\G$ of polynomial growth, the inclusion $\ell^1\G\into\Cred\G$ is dense and spectral.

There is a strong analogy between the results and open questions of Section~\ref{Gelfand section}, and the results and open questions from this section. To start off, we have the following correspondent of Novodvorskii's theorem: a dense and spectral morphism $A\To B$ induces an isomorphism $K_*(A)\to K_*(B)$ (Karoubi \cite[p.109]{Kar78}, Swan \cite[Thms.2.2 \& 3.1]{Swa77}, Connes \cite[VI.3]{Con81}, Bost \cite[Appendix]{Bos90}; see also \cite[Cor.21 \& Prop.46]{Nic08} for a generalization).

We pass to stable ranks, where the following lemma is useful:

\begin{lem}\label{recognizeunimodular}
 Let $\phi:A\To B$ be a dense and spectral morphism. Then, for all $\underline{a}\in A^n$ we have that $\underline{a}\in \Lg_n(A)$ if and only if $\phi(\underline{a})\in \Lg_n(B)$. In particular, $\phi(\Lg_n(A))$ is dense in $\Lg_n(B)$.
\end{lem}
\begin{proof} Let $(\phi(a_i))\in \Lg_n(B)$. Thus $\sum b_i\phi(a_i)\in B^\times$ for some $(b_i)\in B^n$. The density of $\phi(A)$ in $B$ allows us to assume that $b_i=\phi(a_i')$ with $a_i'\in A$. Then $\phi(\sum a_i'a_i)\in B^\times$. As $\phi$ is spectral, we obtain $\sum a_i'a_i\in A^\times$, so $(a_i)\in \Lg_n(A)$. The other implication is trivial. As for the second part, $\phi(A^n)$ is dense in $B^n$ so $\phi(A^n)\cap \Lg_n(B)=\phi(\Lg_n(A))$ is dense in $\Lg_n(B)$. \end{proof}

\begin{thm}\label{bsrspectral} Let $\phi: A\To B$ be a dense and spectral morphism. Then $\bsr A \leq\bsr B$.
\end{thm}

This result is due to Swan \cite[Thm.2.2 c)]{Swa77}; compare Theorem~\ref{Gelfand bsr}. Here is a short argument, different from Swan's.

\begin{proof}
Put $n=\bsr B$, and let $(a_i)\in \Lg_{n+1}(A)$. Then $(\phi(a_i))\in \Lg_{n+1}(B)$, so there is $(b_i)\in B^n$ such that $(\phi(a_i)+b_i\phi(a_{n+1}))\in \Lg_n(B)$. As $\phi(A)$ is dense in $B$ and $\Lg_n(B)$ is open, we may assume that $b_i=\phi(x_i)$ for some $x_i\in A$. Thus $(\phi(a_i+x_ia_{n+1}))\in \Lg_n(B)$, hence $(a_i+x_ia_{n+1})\in \Lg_n(A)$ by Lemma~\ref{recognizeunimodular}. We conclude that $n\geq\bsr A$.
\end{proof}

A notable result addressing Swan's problem for the Bass stable rank, due to Badea \cite[Thm.1.1]{Bad98}, says the following: if $A$ is a dense and spectral Banach $*$-subalgebra of a $\Cstar$-algebra $B$, and if $A$ is closed under $C^\infty$-functional calculus for self-adjoint elements, then $\bsr A=\bsr B$. This applies, for instance, to dense subalgebras coming from derivations (\cite[Cor.4.10]{Bad98}). Note that solving Problem~\ref{bsrquestion} would solve Swan's problem for the Bass stable rank, as well. Note also that $\tsr A \geq\tsr B \geq\bsr B \geq\bsr A$ whenever $A\To B$ is a dense and spectral morphism (the first inequality by Theorem~\ref{densetsr}, the second inequality holds in general, and the last inequality by Theorem~\ref{bsrspectral}). Thus, generalizations of Theorem~\ref{HV} to, say, dense and spectral subalgebras of $\Cstar$-algebras would solve Swan's problem for both dimensional stable ranks whenever $A$ is such a subalgebra.

Other than Theorem~\ref{densetsr}, no results pertaining to Swan's problem for the topological stable rank are known, a situation which somewhat  mirrors our ignorance from the Gelfand context (Problem~\ref{Gelfand tsr}). We point out that Badea's results \cite[Thm.4.13]{Bad98} have unnatural hypotheses.

Let us consider now the connected stable rank and the general stable rank. For these, one can give a positive answer to Swan's Problem in full generality (compare Theorem~\ref{Gelfand csrgsr}):

\begin{thm}\label{Swancsrgsr}
Let $\phi:A\To B$ be a dense and spectral morphism. Then $\csr A=\csr B$ and $\gsr A=\gsr B$.
\end{thm}

\begin{proof} 
First,  we note that the proof of Theorem~\ref{densecsrgsr} can be easily adapted to show that $\csr A\geq \csr B$ and $\gsr A\geq \gsr B$. The point is to have $\phi(\Lg_m(A))$ dense in $\Lg_m(B)$; in the proof of Theorem~\ref{densecsrgsr} this was guaranteed as soon as $m\geq \tsr A$, whereas here it holds for all $m$ according to Lemma~\ref{recognizeunimodular}.

We claim that $\gsr B\geq \gsr A$. We let $m\geq\gsr B$, and we show that each unimodular $m$-tuple over $A$ is the last column of a matrix in $\GL_m(A)$; this means that $\GL_m(A)$ acts transitively on $\Lg_m(A)$, which then leads to $\gsr B\geq \gsr A$. Let $\underline{a}\in \Lg_m(A)$. Then $\phi(\underline{a})\in \Lg_m(B)$, so - by the transitivity of the action of $\GL_m(B)$ on $\Lg_m(B)$ - there is a matrix $\beta\in \GL_m(B)$ having $\phi(\underline{a})$ as its last column. As $\phi(A)$ is dense in $B$, we can approximate the entries of $\beta$, except for the last column, so as to get a matrix $\beta'\in\GL_m(B)$ which has all its entries in $\phi(A)$, and still has $\phi(\underline{a})$ as its last column. Put $\beta'=\phi(\alpha)$, where $\alpha\in\M_m(A)$ has $\underline{a}$ as its last column. We now invoke the following fact (Swan \cite[Lem.2.1]{Swa77}; see also Bost \cite[Prop.A.2.2]{Bos90} and Schweitzer \cite[Thm.2.1]{Sch92}): 

\begin{quote} $(*)$ If $A\to B$ is a dense and spectral morphism, then $\M_m(A)\to\M_m(B)$ is a dense and spectral morphism for each $m\geq 1$. 
\end{quote}
This fact tells us that $\alpha\in\GL_m(A)$, which ends the proof our claim that $\gsr B\geq \gsr A$. 

Next, we claim that $\csr B\geq \csr A$. The proof is very similar to the one for the general stable rank. We let $m\geq\csr B$, and we show that each unimodular $m$-tuple over $A$ is the last column of a matrix in $\GL^0_m(A)$. Let $\underline{a}\in \Lg_m(A)$. Since $\GL^0_m(B)$ acts transitively on $\Lg_m(B)$, there is a matrix $\beta\in \GL^0_m(B)$ having $\phi(\underline{a})\in \Lg_m(B)$ as its last column. The density of $\phi(A)$ in in $B$ allows us to replace $\beta$ by a matrix $\beta'\in \GL^0_m(B)$ whose entries are in $\phi(A)$ and which still has $\phi(\underline{a})$ as the last column. Put $\beta'=\phi(\alpha)$, where $\alpha\in\M_m(A)$ has $\underline{a}$ as its last column. As soon as we show that $\alpha\in\GL^0_m(A)$,  our claim is proved. So let us prove the following fact:

\begin{quote} $(**)$ If $\phi:A\To B$ is a dense and spectral morphism, then for all $\alpha\in \M_m(A)$ we have that $\alpha\in\GL^0_m(A)$ if and only if $\phi(\alpha)\in\GL^0_m(B)$.
\end{quote}

The forward direction is obvious; we argue the converse. Due to $(*)$, it suffices to consider the case $m=1$. Let $a, a'\in A$ with $\phi(a), \phi(a')$ lying in the same component of $B^\times$. First of all, $a$ and $a'$ are in $A^\times$. Let $p:[0,1]\to B^\times$ be a path from $\phi(a)=p(0)$ to $\phi(a')=p(1)$. For each $t\in [0,1]$, let $V_t$ be an open, convex neighborhood of $p(t)$ contained in $B^\times$. Let $0=t_0<t_1<\dots<t_k=1$ be such that $\{V_{t_j}\}_{0\leq j\leq k}$ is an open cover of $p([0,1])$. Connectivity of $p([0,1])$ tells us that we can extract a sub-index set $0=s_0<s_1<\dots<s_l=1$ such that $V_{s_{j-1}}$ meets $V_{s_j}$ for $1\leq j\leq l$. As $\phi(A)$ is dense in $B$, we can pick $x_j\in A$ such that $\phi(x_j)\in V_{s_{j-1}}\cap V_{s_j}$ for $1\leq j\leq l$. Let $q_A$ be the broken line from $x_0=a$ to $x_{l+1}=a'$ with successive vertices $x_j$. Then $q_B:=\phi(q_A)$, the broken line from $\phi(a)$ to $\phi(a')$ with successive vertices $\phi(x_j)$, lies in $B^\times$ since each line segment from $\phi(x_{j-1})$ to $\phi(x_j)$ lies in the convex set $V_{s_{j-1}}$. Hence $q_A$ lies entirely in $A^\times$, showing that $a$ and $a'$ are in the same component of $A^\times$.
\end{proof}

\begin{rem}
The csr half of Theorem~\ref{Swancsrgsr} generalizes \cite[Thm.4.15]{Bad98} by removing the commutativity assumption. It was first proved in \cite[Prop.36]{Nic08} in the context of relatively spectral morphisms. This is a weaker notion of spectral morphism, in which the spectral invariance is only known over a dense subalgebra; specifically, a morphism $\phi:A\To B$ is said to be \emph{relatively spectral} if $\spec_B (\phi(x))=\spec_A (x)$ for all $x$ in a dense subalgebra of $A$ (\cite[Def.10]{Nic08}). The argument given above is different, and is motivated by the analogy between the general and the connected stable ranks we have been following throughout the paper. 

The proof of Theorem~\ref{Swancsrgsr} can be easily adapted to yield the following stronger statement: if $\phi:A\To B$ is a dense and completely relatively spectral morphism, then $\csr A=\csr B$ and $\gsr A=\gsr B$. We refer to \cite[Def.13]{Nic08} for the definition of a completely relatively spectral morphism; informally, all matrix amplifications of such a morphism are relatively spectral. An example of a dense and completely relatively spectral morphism is the inclusion $\ell^1\G\into\Cred\G$ for $\G$ a finitely generated group of subexponential growth (\cite[Ex.49]{Nic08}).
\end{rem}

\begin{rem} In Section~\ref{Gelfand section}, we argued that $\gsr A=\gsr C(X_A)$ by invoking the Forster - Taylor theorem, which says that the Gelfand transform $A\To C(X_A)$ induces a monoid isomorphism $\mathcal{P}(A)\to\mathcal{P}(C(X_A))$. Here, a similar fact holds (Bost \cite[A.2]{Bos90}): a dense and spectral morphism $A\to B$ induces a monoid isomorphism $\mathcal{P}(A)\to\mathcal{P}(B)$. This gives an alternate way of proving the invariance of the general stable rank from Theorem~\ref{Swancsrgsr}. 

However, the direct proof given above has the advantage of being ring-theoretic. To explain what we mean, consider the following setting (conditions (1), (2), and (3') of \cite{Swa77}):
\begin{itemize}
\item[$\cdot$] $A$ is a unital ring;

\item[$\cdot$] $B$ is a unital topological ring with the property that the invertible group $B^\times$ is open, and the inversion $u\mapsto u^{-1}$ is continuous on $B^\times$;

\item[$\cdot$] $\phi:A\to B$ is a unital ring morphism with dense image and with the property that $a\in A$ is invertible in $A$ if and only if $\phi(a)$ is invertible in $B$. 
\end{itemize}
\noindent Then the gsr half of the proof of Theorem~\ref{Swancsrgsr} actually shows that $\gsr A=\gsr B$. On the other hand, in this ring-theoretic context it is not true, in general, that $\phi$ induces a monoid isomorphism $\mathcal{P}(A)\to\mathcal{P}(B)$ (\cite{Swa77}, start of $\S$3, and Remark 2 on p.213).
\end{rem}

\begin{ex}\label{polygrowth}
Let $\G$ be a finitely generated group of polynomial growth. The inclusion $\ell^1\G\into\Cred\G$ being dense and spectral, we have $\csr \ell^1\G=\csr \Cred\G$ and $\gsr \ell^1\G=\gsr \Cred\G$.

We also have $\bsr \ell^1\G=\bsr \Cred\G$. Indeed, let $L:\ell^2\G\to\ell^2\G$ be the closed, densely defined linear map given by $L(\delta_g)=(1+|g|)\delta_g$, where $|\cdot|$ is a fixed word-length on $\G$. We obtain a closed, unbounded derivation $\delta_L: \Cred\G\to\Cred\G$ defined by $\delta_L(a)=[a,L]$. For all positive integers $k$, the inclusion $\dom(\delta_L^k)\into \Cred\G$ is dense and spectral (see proofs of Corollaries 4.10 and 4.11 in \cite{Bad98} and references therein). On the other hand, it can be checked that, for $\sum a_gg\in \dom(\delta_L^k)$, we have 
\[\delta_L^k\big(\sum a_gg\big)(\delta_1)=\sum a_g|g|^k\delta_g\in \ell^2\G\]
from which we obtain that $\dom(\delta_L^k)\subseteq H^k\G$ for all positive integers $k$. What we said so far works for any finitely generated group $\G$. If $\G$ has polynomial growth, then $\sum (1+|g|)^{-2k}$ converges for $k$ sufficiently large, and from the Cauchy -Schwarz inequality
\[\sum |a_g|\leq \Big(\sum |a_g|^2(1+|g|)^{2k}\Big)^{1/2}\Big(\sum (1+|g|)^{-2k}\Big)^{1/2}\]
we infer that $H^k\G\subseteq \ell^1\G$ for $k$ sufficiently large. Summarizing, we have a chain of dense and spectral inclusions $\dom(\delta_L^k)\into \ell^1\G\into\Cred\G$ for $k$ sufficiently large. From Theorem~\ref{bsrspectral}, we obtain $\bsr (\dom(\delta_L^k))\leq \bsr \ell^1\G\leq \bsr\Cred\G$. By a result of Badea \cite[Cor.4.10]{Bad98}, we have $\bsr (\dom(\delta_L^k))=\bsr\Cred\G$; hence $\bsr \ell^1\G=\bsr \Cred\G$ as well. 

The equality $\tsr \ell^1\G=\tsr \Cred\G$ is very likely to hold, but we do not have a proof. When $\G\simeq\Z^d$, this is confirmed in Example~\ref{Z^d}.
\end{ex}

\begin{rem} It is also likely that, in general, the homotopical stable ranks of $\ell^1\G$ equal the corresponding stable ranks of $\Cred\G$. One is led to such a conjecture not so much by the empirical evidence presented by Example ~\ref{polygrowth}, but rather by the K-theoretic conjecture - sometimes attributed to J.-B. Bost - that $K_*(\ell^1\G)\simeq K_*(\Cred\G)$ for all discrete, countable groups $\G$.
\end{rem}

\section{Inductive limits}\label{Inductive limits}
For the remainder of the paper, Banach algebras are no longer required to be unital.

Following \cite[\S 3.3]{Bla98}, we recall the definition of the inductive limit in the context of Banach algebras. Let $\{A_i\}_{i\in I}$ be an inductive system of Banach algebras, indexed by a directed set $I$. As part of the data, we are given a (not necessarily unital) connecting morphism $\phi_{ij}: A_i\to A_j$ for each $i<j$, in such a way that the following coherence condition is satisfied: $\phi_{ik}=\phi_{jk}\circ\phi_{ij}$ whenever $i<j<k$. The inductive system $\{A_i\}_{i\in I}$ is \emph{normed} if $\limsup_j\|\phi_{ij}(a_i)\|_j<\infty$ for all $i\in I$ and $a_i\in A_i$; note that, in the $\Cstar$-subcontext, this condition is automatic. If $\{A_i\}_{i\in I}$ is a normed inductive system, then the algebraic inductive limit  can be turned into a Banach-algebraic inductive limit as follows: define an obvious seminorm, quotient by the degenerate ideal of the seminorm, and complete. Let $A:=\varinjlim\: A_i$ denote the Banach algebra thus obtained. For each $i\in I$ there is a canonical morphism $\phi_i: A_i\to A$ such that $\phi_i=\phi_j\circ\phi_{ij}$ whenever $i<j$. Furthermore, the directed union $\displaystyle \cup_{i\in I}\;\phi_i(A_i)$ is dense in $A$.

Up to adding a new unit to $A$ and each $A_i$ - which does not affect the stable ranks - we may assume that $A$, each $A_i$, and each $\phi_{ij}$, are unital. 

\begin{lem} The directed union $\displaystyle \cup_{i\in I}\; \phi_i(\Lg_m(A_i))$ is dense in $\Lg_m(A)$ for each $m\geq 1$.
\end{lem}

\begin{proof} Fix $(a_1, \dots, a_m)\in \Lg_m(A)$, and let $b_1, \dots, b_m\in A$ such that $b_1a_1+\ldots +b_ma_m=1$ in $A$. Also, fix an $\e>0$. For some $i\in I$, we may pick $a^{(i)}_1, \dots, a^{(i)}_m$ and $ b^{(i)}_1, \dots, b^{(i)}_m$ in $A_i$ such that $\phi_i(a^{(i)}_1, \dots, a^{(i)}_m)$ is within $\e$ of $(a_1, \dots, a_m)$, and $\phi_i(b^{(i)}_1a^{(i)}_1+\ldots +b^{(i)}_ma^{(i)}_m)$ is close enough to $b_1a_1+\ldots +b_ma_m=1$ as to remain invertible in $A$. By \cite[Lem.3.3.1]{Bla98} we have that $\phi_{ij}(b^{(i)}_1a^{(i)}_1+\ldots +b^{(i)}_ma^{(i)}_m)$ is invertible in $A_j$ for some $j>i$. Then $\phi_{ij}(a^{(i)}_1, \dots, a^{(i)}_m)\in \Lg_m(A_j)$, hence $\phi_i(a^{(i)}_1, \dots, a^{(i)}_m)=\phi_j\big(\phi_{ij}(a^{(i)}_1, \dots, a^{(i)}_m)\big)\in \phi_j(\Lg_m(A_j))$.
\end{proof}

\begin{thm}\label{limits} We have $\sr A\leq \liminf \;\sr A_i$ for $\mathrm{sr}\in\{\mathrm{tsr}, \mathrm{csr}, \mathrm{gsr}\}$.
\end{thm}
In the $\Cstar$-setting, Theorem~\ref{limits} is due to Rieffel \cite[Thm.5.1]{Rie83} for the topological stable rank, and to Nistor \cite[(1.6)]{Nis86} for the connected stable rank. 

\begin{proof} If $\liminf \;\sr A_i$ is infinite, there is nothing to prove; so let $n=\liminf \;\sr A_i$. Then $\sr A_i=n$ for all $i$ in a cofinal subset $I_0$ of $I$. As $I_0$ is cofinal, any directed union indexed by $I$ equals the directed sub-union indexed by $I_0$, e.g., $\displaystyle \cup_{i\in I}\; \phi_i(\Lg_m(A_i))=\displaystyle \cup_{i\in I_0}\; \phi_i(\Lg_m(A_i))$.

We analyze the stable ranks one by one.

Let sr be the topological stable rank. For each $i\in I_0$, $\Lg_n(A_i)$ is dense in $(A_i)^n$, so $\phi_i(\Lg_n(A_i))$ is dense in $(\phi_i(A_i))^n$. Now the density of $\displaystyle \cup_{i\in I_0}\; \phi_i(\Lg_n(A_i))$ in $\displaystyle \cup_{i\in I_0}\;(\phi_i(A_i))^n$ implies the density of $\Lg_n(A)$ in $A^n$.

Let sr be the connected stable rank, and let $m\geq n$. Since the action of $\GL_m^0(A)$ on $\Lg_m(A)$ has open orbits, it suffices to show that $\displaystyle \cup_{i\in I_0}\; \phi_i(\GL_m^0(A_i))$ acts transitively on $\displaystyle \cup_{i\in I_0}\; \phi_i(\Lg_m(A_i))$ in order to conclude that $\GL_m^0(A)$ acts transitively on $\Lg_m(A)$. This is immediate: any two points in $\displaystyle \cup_{i\in I_0}\; \phi_i(\Lg_m(A_i))$ may be assumed to lie in $\phi_i(\Lg_m(A_i))$ for some $i\in I_0$, and $\phi_i(\GL^0_m(A_i))$ acts transitively on $\phi_i(\Lg_m(A_i))$.

Let sr be the general stable rank. The action of $\GL_m(A)$ on $\Lg_m(A)$ also has open orbits, so the argument for the connected stable rank applies - \emph{mutatis mutandis} - to the general stable rank, as well.
\end{proof}

We do not know whether Theorem~\ref{limits} holds for the Bass stable rank; this problem, recorded below, is related to Problem~\ref{bsrquestion}.

\begin{prob} 
Does $\bsr A\leq \liminf \;\bsr A_i$ hold?
\end{prob}

\begin{ex}
Let $A$ be an AF $\Cstar$-algebra (e.g. $\KH$, the $\Cstar$-algebra of compact operators on an infinite-dimensional, separable Hilbert space). A finite-dimensional $\Cstar$-algebra has all stable ranks equal to $1$, and the property of having all stable ranks equal to $1$ is preserved under inductive limits. Hence $\sr A=1$. 
\end{ex}

\begin{rem} Let $A$ be a unital $\Cstar$-algebra. The inequality $\sr (A\otimes \mathcal{K})\leq  \liminf \;\sr \M_n(A)$, stipulated by Theorem~\ref{limits}, can be very strict. On the left-hand side, we have $\sr(A\otimes \K)\leq 2$; this is due to Rieffel \cite[Thm.6.4]{Rie83} for the Bass / topological stable rank, and to Nistor \cite[Cor.2.5]{Nis86} and Sheu \cite[Thm.3.10]{She87} for the connected stable rank. But the right-hand side can be infinite, e.g., for the Cuntz algebra $\mathcal{O}_2$. \end{rem}


\section{Extensions} 
Consider a short exact sequence $0\To J\To A\To B\To 0$ of Banach algebras. We have already bounded the stable ranks of $B$ in terms of the stable ranks of $A$ in Theorems~\ref{ontotsrbsr} and ~\ref{ontocsrgsr}. The goal is to bound the stable ranks of $J$ in terms of those of $A$, and the stable ranks of $A$ in terms of those for $J$ and $B$. In some of the results below, we need the closed ideal $J$ to have a \emph{bounded approximate identity}. Recall, a bounded approximate identity for $J$ is a uniformly bounded net $(j_\alpha)\subseteq J$ such that $j_\alpha j\To j$ and $jj_\alpha\To j$ for all $j\in J$. In the $\Cstar$-setting, this is automatic: every closed ideal in a $\Cstar$-algebras has a bounded approximate identity.

Up to forced unitization, we may assume that both $A$ and $B$ are unital. Let 
\[J^+=\{\lambda+j\: : \: \lambda\in\C, j\in J\}\]
be the unital Banach subalgebra of $A$ obtained by adjoining the unit of $A$ to $J$. The (closed) inclusion $J^+\into A$ is spectral: if $(\lambda+j)a=1$ for some $a\in A$, then $\lambda\neq 0$ and $a=\frac{1}{\lambda}(1-ja)\in J^+$.

We first prove a general lemma that will help us recognize unimodular vectors over $J^+$:

\begin{lem} Assume $J$ has an approximate identity. Then $\Lg_n(J^+)=\Lg_n(A)\cap (J^+)^n$.
\end{lem}

\begin{proof} For the non-trivial inclusion, let $(\lambda_i+j_i)\in (J^+)^n\cap \Lg_n(A)$, where $\lambda_i\in\C$ and $j_i\in J$, and let $(a_i)\in A^n$ with $\sum a_i(\lambda_i+j_i)=1$. In particular, $\lambda_{i_0}\neq 0$ for some $i_0$. 

Let $(j_\alpha)\subseteq J$ be an approximate identity. We look for $(a'_i) \in (J^+)^n$ such that $\sum a'_i(\lambda_i+j_i)$ is close enough to 1 as to make it invertible in $A$. Since $\sum a'_i(\lambda_i+j_i)\in J^+$, it is actually invertible in $J^+$, allowing us to conclude that $(\lambda_i+j_i)\in\Lg_n(J^+)$ as desired.

Put 
\[a'_{i_0}:=a_{i_0}+\sum_{i\neq i_0}\frac{\lambda_i}{\lambda_{i_0}}a_i(1-j_\alpha), \qquad a'_i:=a_ij_\alpha \quad (i\neq i_0)\]
with $\alpha$ still to be chosen. Note that each $a'_i$ is in $J^+$. For $i\neq i_0$ this is obvious; we check that $a'_{i_0}\in J^+$. From $\sum a_i(\lambda_i+j_i)=1$ we deduce that $\sum \lambda_ia_i\in 1+J$, so we obtain
\[a'_{i_0}=a_{i_0}+\sum_{i\neq i_0}\frac{\lambda_i}{\lambda_{i_0}}a_i(1-j_\alpha)\in \Big(a_{i_0}+\sum_{i\neq i_0}\frac{\lambda_i}{\lambda_{i_0}}a_i \Big)+J=\frac{1}{\lambda_{i_0}}\big(\sum \lambda_ia_i\big)+J\subseteq J^+.\]
On the other hand, one computes
\begin{eqnarray*}
\sum a'_i(\lambda_i+j_i)=\big(\sum \lambda_i a_i\big)\Big(1+\frac{j_{i_0}}{\lambda_{i_0}}\Big)+\sum_{i\neq i_0} a_ij_\alpha\Big(j_i-\frac{\lambda_i}{\lambda_{i_0}}j_{i_0}\Big)
\end{eqnarray*}
which converges to
\[\big(\sum \lambda_i a_i\big)\Big(1+\frac{j_{i_0}}{\lambda_{i_0}}\Big)+\sum_{i\neq i_0} a_i\Big(j_i-\frac{\lambda_i}{\lambda_{i_0}}j_{i_0}\Big)=a_{i_0}(\lambda_{i_0}+j_{i_0})+\sum_{i\neq i_0} a_i(\lambda_i+j_i)=1.\]
Thus, we pick $\alpha$ such that $\sum a'_i(\lambda_i+j_i)$ is invertible in $A$. This ends the proof.
\end{proof}

For the dimensional stable ranks, we can estimate the stable rank of $J$ in terms of the stable rank of $A$. The next result is due to Vaserstein \cite[Thm.4]{Vas71} for the Bass stable rank, and to Rieffel \cite[Thm.4.4]{Rie83} for the topological stable rank.

\begin{thm} Let $J$ be a closed ideal in $A$. Then $\bsr J \leq \bsr A$. If $J$ has a bounded approximate identity, then $\tsr J\leq \tsr A$.
\end{thm}

\begin{rem} For the homotopical stable ranks such a result is not true, that is, neither $\csr J\leq \csr A$ nor $\gsr J\leq \gsr A$ hold in general. Consider for instance the closed ideal $C_0(I^d\setminus \partial I^d)$ of $C(I^d)$. The unitization of $C_0(I^d\setminus \partial I^d)$ is isomorphic to $C(S^d)$, so both $\csr C_0(I^d\setminus \partial I^d)$ and $\gsr C_0(I^d\setminus \partial I^d)$ are at least $d/2$ when $d>4$. On the other hand, $\csr C(I^d)=\gsr C(I^d)=1$ since $I^d$ is contractible.
\end{rem}

Next, we estimate the stable ranks of $A$ in terms of the stable ranks of $J$ and the stable ranks of $B$. Theorem~\ref{extensiontsrbsr} is due to Vaserstein \cite[Thm.4]{Vas71} for the Bass stable rank, and to Rieffel \cite[Thm.4.11]{Rie83} for the topological stable rank. The connected stable rank estimate of Theorem~\ref{extensioncsrgsr} is due to Nagy \cite[Lem.2]{Nag85} and independently to Sheu \cite[Thm.3.9]{She87}. We observe that a general rank estimate can be established in the same way. 

\begin{thm}\label{extensiontsrbsr} Let $0\To J\To A\To B\To 0$ be an exact sequence of Banach algebras. Then: 
\[\tsr A \leq \max\{\tsr J ,\tsr B ,\csr B \}, \qquad \bsr A  \leq \max\{\bsr J ,\bsr B +1\}\]
\end{thm}

\begin{thm}\label{extensioncsrgsr} Let $0\To J\To A\To B\To 0$ be an exact sequence of Banach algebras, and assume that $J$ has an approximate identity. Then: 
\[ \csr A \leq \max\{\csr J ,\csr B\}, \qquad \gsr A  \leq \max\{\gsr J ,\csr B \}\]
\end{thm}

\begin{proof} Let $\pi: A\to B$ denote the quotient map.

Let $m\geq \max\{\csr J,\csr B\}$; we show that $\GL^0_m(A)$ acts transitively on $\Lg_m(A)$. Let $\underline{a}\in\Lg_m(A)$, so $\pi(\underline{a})\in\Lg_m(B)$. As $m\geq \csr B$, there is $\beta\in\GL^0_m(B)$ such that $\beta\cdot \pi(\underline{a})^T=(1,0,\dots,0)^T$. Since $\pi: \GL^0_m(A)\to\GL^0_m(B)$ is onto, there is $\alpha\in \GL^0_m(A)$ with $\pi(\alpha)=\beta$ and so $\alpha\cdot \underline{a}^T=(j_1+1,j_2,\dots,j_m)^T$ for some $(j_i)\in J^m$. It follows that 
$(j_1+1,j_2,\dots,j_m)\in\Lg_m(A)\cap(J^+)^m=\Lg_m(J^+)$. As $m\geq \csr J$, there is $\mu\in\GL^0_m(J^+)$ such that $\mu\cdot(j_1+1,j_2,\dots,j_m)^T=(1,0,\dots,0)^T$. Thus,  $\mu\alpha$ takes $\underline{a}$ to $(1,0,\dots,0)$.

To show $\gsr A \leq \max\{\gsr J ,\csr B \}$, the steps are the same up to the appearance of $\mu$. In this case, $\mu$ is in $\GL_m(J^+)$, and the conclusion is that $\GL_m(A)$ acts transitively on $\Lg_m(A)$.
\end{proof}

Theorems ~\ref{extensiontsrbsr} and ~\ref{extensioncsrgsr}, together with Theorems ~\ref{ontotsrbsr} and ~\ref{ontocsrgsr}, are quite effective for computing the stable ranks of $\Cstar$-extensions of $\K$ by $C(X)$, and even of tensor products of such Toeplitz-like $\Cstar$-algebras (see Section~\ref{higherconnectedextensions}). We start with the simplest example (compare \cite[Ex.4.13]{Rie83}):

\begin{ex}\label{toep} The Toeplitz $\Cstar$-algebra $\mathcal{T}$, the $\Cstar$-algebra generated by a non-unitary isometry, fits into an extension $0\To\K\To\mathcal{T}\To C(S^1)\To 0$.
Therefore:
\[\tsr \mathcal{T}  \leq \max\{\tsr \mathcal{K} ,\tsr  C(S^1) ,\csr  C(S^1)\},\qquad \csr \mathcal{T}  \leq \max\{\csr \mathcal{K} ,\csr C(S^1)\}\]
We know that $\tsr \mathcal{K} =\csr \mathcal{K} =1$, $\tsr C(S^1) =1$ and $\csr C(S^1)=2$. It follows that $\tsr \mathcal{T} \leq 2$ and $\gsr \mathcal{T}\leq\csr \mathcal{T} \leq 2$. As $\mathcal{T}$ is infinite, we conclude that $\sr \mathcal{T}=2$.
\end{ex}

\begin{rem} For an extension $0\To J\To A\To B\To 0$ of Banach algebras, the (expected) inequality $\sr A\leq \max\{\sr J,\sr B\}$ holds when sr is the connected stable rank. The Toeplitz algebra extension shows that this is no longer true, in general, for any one of the remaining three stable ranks (topological, Bass, and general).
\end{rem}

\begin{ex}\label{hightoep} Let $n> 1$. For $\mathcal{T}_n$, the Toeplitz $\Cstar$-algebra on the odd-dimensional sphere $S^{2n-1}$ (see Coburn \cite{Cob73}), we have the corresponding extension $0\To\KH\To \mathcal{T}_n\To C(S^{2n-1})\To 0$.

The fact that $\tsr \mathcal{T}_n=n$ follows from a result of Nistor \cite[Thm.4.4]{Nis86}; see Theorem~\ref{stablesimple} below. For the connected stable rank, recall the estimates from Theorem~\ref{extensioncsrgsr} and Theorem~\ref{ontocsrgsr}:
\[\csr  \mathcal{T}_n\leq \max\{\csr \K, \csr C(S^{2n-1})\}, \qquad \csr C(S^{2n-1})\leq \max\{\csr \mathcal{T}_n, \tsr \mathcal{T}_n\}\]
As $\csr \K=1$, $\csr C(S^{2n-1})=n+1$, and $\tsr \mathcal{T}_n=n$, it follows that $\csr \mathcal{T}_n=n+1$. We now show that $\gsr \mathcal{T}_n=n+1$. First, note that $\gsr \mathcal{T}_n\leq n+1$ from the computation of $\csr\mathcal{T}_n$. We also have
\[\gsr C(S^{2n-1})\leq \max\{\gsr \mathcal{T}_n, \tsr \mathcal{T}_n\}\]
by Theorem~\ref{ontocsrgsr}. For $n>2$, we know that $\gsr C(S^{2n-1})=n+1$ (Proposition~\ref{gsrspheres}); then $\tsr \mathcal{T}_n=n$ forces $\gsr \mathcal{T}_n=n+1$. For $n=2$ we have $\gsr C(S^3)=1$, which no longer implies that $\gsr \mathcal{T}_2=3$. Nevertheless, we know that $\gsr \mathcal{T}_2\leq 3$, and we recall that $\mathcal{T}_2$ is finite but not stably finite (see \cite[6.10.1]{Bla98}). If $\gsr \mathcal{T}_2$ were at most 2, then the finiteness of $\mathcal{T}_2$ would actually imply $\gsr \mathcal{T}_2=1$, which in turn would imply that $\mathcal{T}_2$ is stably finite - a contradiction. Thus $\gsr \mathcal{T}_2=3$.

We conclude that $\mathcal{T}_n$ has the dimensional stable ranks equal to $n$, and the homotopical stable ranks equal to $n+1$.
\end{ex}

\section{Tensor products of $\Cstar$-extensions of $\KH$ by commutative $\Cstar$-algebras}\label{higherconnectedextensions}
Consider the following set-up:
\begin{itemize}
\item[($\ddagger$)] For $1\leq i\leq n$, let $X_i$ be a compact metric space, and let $A_i$ be a unital $\Cstar$-extension of $\KH$ by $C(X_i)$. Put $A:=A_1\otimes\dots \otimes A_n$, and $X:=X_1\times\dots\times X_n$.
\end{itemize}
Each $A_i$ is nuclear (see \cite[Thm.15.8.2]{Bla98}), so we do not need to specify which $\Cstar$-tensor product we are using. However, for the purposes of Lemma~\ref{basicstep} below, it is convenient to agree that $\otimes$ stands for the \emph{maximal} tensor product in what follows.

The main result of Nistor's paper \cite{Nis86} is the computation of the dimensional stable rank for such tensor products:

\begin{thm}\label{stablesimple} Keep the notations of $(\ddagger)$, and assume $\dim X\neq 1$. Then $\tsr A=\tsr C(X)$.
\end{thm}
Note that the dimensional assumption is not superfluous: for the Toeplitz $\Cstar$-algebra $\mathcal{T}$ we have $\tsr \mathcal{T}=2$, whereas $\tsr C(S^1)=1$.

Under the hypothesis of Theorem~\ref{stablesimple}, we also have $\csr A\leq \csr C(X)$ (\cite[Prop.3.4]{Nis86}). Nistor proves these stable rank results under the assumption that each compact space $X_i$ can be realized as the inverse limit of finite CW-complexes of dimension $\dim X_i$ (cf. assumptions before Lemma 3.7 in \cite{Nis86}); he then points out that the assumption on $X_i$ is fulfilled whenever $X_i$ is a compact manifold. It is actually the case that $X_i$ is the inverse limit of a sequence of finite CW-complexes of dimension $\dim X_i$ whenever $X_i$ is a compact metric space. This follows by combining two ingredients: Freudenthal's theorem \cite{Fre37} that every compact metric space of dimension $\leq n$ is the inverse limit of a sequence of finite CW-complexes of dimension $\leq n$, and the well-known fact that the inverse limit of a sequence of compact spaces of dimension $\leq n$ is a compact space of dimension $\leq n$. Consequently, Nistor's results are indeed available in the generality of $(\ddagger)$.

The goal of this section is to show that the homotopical stable ranks of the $\Cstar$-algebra $A$ can be computed in certain favorable circumstances: 

 \begin{thm}\label{csrgsrtoeplitztensorsimple} Keep the notations of $(\ddagger)$, and assume $\dim X\neq 1$. 
 
 a) If $\csr C(X)>\tsr C(X)$, then $\csr A=\csr C(X)$.
 
 b) If $\gsr C(X)>\tsr C(X)$, then $\csr A=\csr C(X)$ and $\gsr A=\gsr C(X)$.
 \end{thm}

Roughly speaking, both Theorems~\ref{stablesimple} and ~\ref{csrgsrtoeplitztensorsimple} can be summarized under the slogan that $A$ and its ``symbol algebra'' $C(X)$ have the same stable ranks. Theorem~\ref{csrgsrtoeplitztensorsimple}, however, needs fairly strong assumptions on the symbol algebra. For a finite-dimensional compact space $Y$, the property that $\gsr C(Y)>\tsr C(Y)$, respectively that $\csr C(Y)>\tsr C(Y)$, is equivalent to having $\gsr C(Y)$, respectively $\csr C(Y)$, achieve the dimensional upper bound $\tsr C(Y)+1$ (cf. Theorem~\ref{ord}); one can think of such a space $Y$ as being ``gsr-full'', respectively ``csr-full''. Since $\gsr\leq \csr$, if $Y$ is gsr-full then $Y$ is csr-full. Theorems~\ref{bsrtsr from dimension}, ~\ref{csrbound} and ~\ref{csrdim} show that $Y$ is csr-full if and only if $Y$ is odd-dimensional with non-vanishing top cohomology. In what concerns gsr-fullness, recall that spheres in odd dimensions $\geq 5$ are gsr-full (cf. Proposition~\ref{gsrspheres}).

We also point out that the relation between the stable ranks of a tensor product and the corresponding stable ranks of the factors is poorly understood. In particular, one cannot reduce the computation of the homotopical stable ranks of $A$ to the corresponding computation  for each of the $A_i$'s.

We now proceed to the proof of Theorem~\ref{csrgsrtoeplitztensorsimple}. The first step is the following

\begin{lem}\label{basicstep} Let $0\To \KH\To E\To C(Y)\To 0$ be an exact $\Cstar$-sequence with $E$ unital and $Y$ compact. Then, for each unital $\Cstar$-algebra $D$, we have:
\begin{eqnarray*}
&\tsr D\otimes C(Y)\leq \tsr D\otimes E\leq\big(\tsr D\otimes C(Y)\big)\vee \big(\csr D\otimes C(Y)\big)\\
&\quad\csr D\otimes E\leq \csr D\otimes C(Y)\leq \big(\tsr D\otimes E\big)\vee \big(\csr D\otimes E\big)\\ 
&\gsr D\otimes C(Y)\leq \big(\tsr D\otimes E\big)\vee \big(\gsr D\otimes E\big)
\end{eqnarray*}
\end{lem}

The proof of Lemma~\ref{basicstep} uses the following general fact: if $D$ is a unital $\Cstar$-algebra and $X$ is a compact space, then $\sr D\otimes \K\leq \sr D\leq \sr D\otimes C(X)$. The first inequality follows by combining the estimates for matrix algebras (Theorems~\ref{matrixtsrbsr} and ~\ref{matrixcsrgsr}) and inductive limits (Theorem~\ref{limits}). As for the second inequality, it follows from Theorem~\ref{ontotsrbsr} for the dimensional stable ranks, and from Proposition~\ref{ontosplit} for the homotopical stable ranks.

\begin{proof} Consider the exact sequence $0\To D\otimes \K\To D\otimes E\To D\otimes C(Y)\To 0$. On the one hand, the behavior of stable ranks with respect to quotients yields the following estimates:
\begin{eqnarray*}
&\tsr D\otimes C(Y)\leq \tsr D\otimes E \\
&\csr D\otimes C(Y)\leq\big( \tsr D\otimes E\big)\vee \big(\csr D\otimes E\big)\\
&\gsr D\otimes C(Y)\leq\big(\tsr D\otimes E\big)\vee \big(\gsr D\otimes E\big)
\end{eqnarray*}
On the other hand, by the behavior of stable ranks with respect to extensions we have:
\begin{eqnarray*}
& \tsr D\otimes E\leq \big(\tsr D\otimes \K\big)\vee \big(\tsr D\otimes C(Y)\big)\vee \big(\csr D\otimes C(Y)\big)\\
&\csr D\otimes E\leq \big(\csr D\otimes \K\big)\vee \big(\csr D\otimes C(Y)\big)
\end{eqnarray*}
Using the fact that $\sr D\otimes \KH\leq \sr D\otimes C(Y)$, the above estimates simplify to
\begin{eqnarray*}
& \tsr D\otimes E\leq \big(\tsr D\otimes C(Y)\big)\vee \big(\csr D\otimes C(Y)\big)\\
&\csr D\otimes E\leq \csr D\otimes C(Y)
\end{eqnarray*}
The proof is complete.
\end{proof}

From this lemma we obtain (compare \cite[Prop.3.4]{Nis86}):

\begin{prop}\label{msr-tsr-csr} Keep the notations of $(\ddagger)$, and let $Z$ be a compact space. Then:
\begin{eqnarray*}
&\tsr C(X\times Z)\leq \tsr A\otimes C(Z)\leq\tsr C(X\times Z)\vee \csr C(X\times Z)\\
&\csr A\otimes C(Z)\leq \csr C(X\times Z)\leq \big(\tsr A\otimes C(Z)\big)\vee \big(\csr A\otimes C(Z)\big)\\
&\gsr C(X\times Z)\leq \big(\tsr A\otimes C(Z)\big)\vee \big(\gsr A\otimes C(Z)\big)
\end{eqnarray*}
\end{prop}

\begin{proof} We argue by induction on $n$. The base case $n=1$ is obtained by setting $E=A_1$, $Y=X_1$ and $D=C(Z)$ in Lemma~\ref{basicstep}. For the induction step, assume the conclusion of the proposition is valid for $n=k$; to show that it holds for $n=k+1$ means to show that the following estimates hold for all compact spaces $Z$: 
\begin{eqnarray}
&\tsr C({\mathcal X}_{k+1}\times Z)\leq \tsr {\mathcal A}_{k+1}\otimes C(Z)\leq \tsr C({\mathcal X}_{k+1}\times Z)\vee \csr C({\mathcal X}_{k+1}\times Z)\\
&\csr {\mathcal A}_{k+1}\otimes C(Z)\leq \csr C({\mathcal X}_{k+1}\times Z)\leq \big(\tsr {\mathcal A}_{k+1}\otimes C(Z)\big)\vee \big(\csr {\mathcal A}_{k+1}\otimes C(Z)\big)\\ 
&\gsr C({\mathcal X}_{k+1}\times Z)\leq \big(\tsr {\mathcal A}_{k+1}\otimes C(Z)\big)\vee \big(\gsr {\mathcal A}_{k+1}\otimes C(Z)\big)
\end{eqnarray}
\[\textrm{where } {\mathcal A}_{k+1}:=\otimes_{i=1}^{k+1} A_i, \quad {\mathcal X}_{k+1}:=\times_{i=1}^{k+1} X_i\]
Fix $Z$. Setting $E=A_{k+1}$, $Y=X_{k+1}$, and $D={\mathcal A}_k\otimes C(Z)$ in Lemma~\ref{basicstep}, we have the following system of inequalities:
\begin{eqnarray*}
&\tsr {\mathcal A}_k\otimes C(X_{k+1}\times Z)\leq \tsr {\mathcal A}_{k+1}\otimes C(Z)\leq\big(\tsr {\mathcal A}_k\otimes C(X_{k+1}\times Z)\big)\vee \big(\csr {\mathcal A}_k\otimes C(X_{k+1}\times Z)\big)\\
&\quad\csr {\mathcal A}_{k+1}\otimes C(Z)\leq \csr {\mathcal A}_k\otimes C(X_{k+1}\times Z)\leq \big(\tsr {\mathcal A}_{k+1}\otimes C(Z)\big)\vee \big(\csr {\mathcal A}_{k+1}\otimes C( Z)\big)\\ 
&\gsr {\mathcal A}_k\otimes C(X_{k+1}\times Z)\leq \big(\tsr {\mathcal A}_{k+1}\otimes C(Z)\big)\vee \big(\gsr {\mathcal A}_{k+1}\otimes C(Z)\big)
\end{eqnarray*}
The induction hypothesis for the compact space $X_{k+1}\times Z$ provides another system of inequalities:
\begin{eqnarray*}
&\tsr C({\mathcal X}_{k+1}\times Z)\leq \tsr {\mathcal A}_k\otimes C(X_{k+1}\times Z)\leq\tsr C({\mathcal X}_{k+1}\times Z)\vee \csr C({\mathcal X}_{k+1}\times Z)\\
&\csr {\mathcal A}_k\otimes C(X_{k+1}\times Z)\leq \csr C({\mathcal X}_{k+1}\times Z)\leq \big(\tsr {\mathcal A}_k\otimes C(X_{k+1}\times Z)\big)\vee \big(\csr {\mathcal A}_k\otimes C(X_{k+1}\times Z)\big)\\ 
&\gsr C({\mathcal X}_{k+1}\times Z)\leq \big(\tsr {\mathcal A}_k\otimes C(X_{k+1}\times Z)\big)\vee \big(\gsr {\mathcal A}_k\otimes C(X_{k+1}\times Z)\big)
\end{eqnarray*}
These two systems of inequalities imply the desired estimates (1) - (3).
 \end{proof}
 
In \cite[Thm.4.4]{Nis86}, Nistor actually proves the following strong version of Theorem~\ref{stablesimple}:

\begin{thm}\label{stable} Keep the notations of $(\ddagger)$, and let $Z$ be a compact space with $\dim (X\times Z)\neq1$. Then $\tsr A\otimes C(Z)=\tsr C(X\times Z)$.
\end{thm}

Combining Theorem~\ref{stable}  and Proposition~\ref{msr-tsr-csr}, we obtain the following consequence:

 \begin{cor}\label{csrgsrtoeplitztensor} Keep the notations of $(\ddagger)$, and let $Z$ be a compact space with $\dim (X\times Z)\neq1$. 
 
 a) If $\csr C(X\times Z)>\tsr C(X\times Z)$, then $\csr A\otimes C(Z)=\csr C(X\times Z)$.
 
 b) If $\gsr C(X\times Z)>\tsr C(X\times Z)$, then $\csr A\otimes C(Z)=\csr C(X\times Z)$ and $\gsr A\otimes C(Z)=\gsr C(X\times Z)$.
 \end{cor}
\begin{proof}
a) If $\csr C(X\times Z)>\tsr C(X\times Z)$, then the inequality
\[\csr A\otimes C(Z)\leq \csr C(X\times Z)\leq \big(\tsr A\otimes C(Z)\big)\vee \big(\csr A\otimes C(Z)\big)\]
forces $\csr (A\otimes C(Z))=\csr C(X\times Z)$. 

b) If $\gsr C(X\times Z)>\tsr C(X\times Z)$, then $\gsr C(X\times Z)=\csr C(X\times Z)>\tsr C(X\times Z)$ by Theorem~\ref{ord}. Part a) yields $\csr A\otimes C(Z)=\csr C(X\times Z)$. Hence $\gsr A\otimes C(Z)\leq \gsr C(X\times Z)$, and the inequality 
\[\gsr C(X\times Z)\leq \big(\tsr A\otimes C(Z)\big)\vee \big(\gsr A\otimes C(Z)\big)\]
leads to $\gsr A\otimes C(Z)=\gsr C(X\times Z)$. 
\end{proof}

Now taking $Z$ to be a singleton, we obtain Theorem~\ref{csrgsrtoeplitztensorsimple}.

\end{document}